\documentclass[12pt]{article}

\usepackage[margin=10pt,font=small,textfont=it,labelfont=bf]{caption}
\usepackage{hyperref}

\usepackage[english]{babel}
\RequirePackage{amssymb}
\RequirePackage{amsfonts}
\RequirePackage{amsmath}
\RequirePackage{graphicx}
\RequirePackage{array}
\RequirePackage{vmargin}
\RequirePackage[mathcal]{euscript}
\RequirePackage{amsthm}
\RequirePackage{stmaryrd}
\RequirePackage{latexsym,amscd}
\RequirePackage{url}
\RequirePackage{dsfont}
\RequirePackage{lmodern}
\RequirePackage{enumitem}

\renewcommand{\emph}{\textbf}

\newcommand{\A}{\mathbb{A}}

\newcommand{\D}{\mathbb{D}}
\newcommand{\E}{\mathbb{E}}
\newcommand{\F}{\mathbb{F}}

\newcommand{\N}{\mathbb{N}}

\renewcommand{\P}{\mathbb{P}}

\newcommand{\R}{\mathbb{R}}

\newcommand{\T}{\mathbb{T}}

\newcommand{\W}{\mathbb{W}}

\newcommand{\Z}{\mathbb{Z}}

\newcommand{\cA}{\mathcal{A}}

\newcommand{\cE}{\mathcal{E}}
\newcommand{\cF}{\mathcal{F}}

\newcommand{\cI}{\mathcal{I}}

\newcommand{\cL}{\mathcal{L}}
\newcommand{\cM}{\mathcal{M}}

\newcommand{\cO}{\mathcal{O}}
\newcommand{\cP}{\mathcal{P}}

\newcommand{\cR}{\mathcal{R}}

\newcommand{\cT}{\mathcal{T}}
\newcommand{\cU}{\mathcal{U}}

\newcommand{\cZ}{\mathcal{Z}}

\renewcommand{\a}{\alpha}
\renewcommand{\b}{\beta}
\newcommand{\g}{\gamma}
\newcommand{\dl}{\delta}
\newcommand{\eps}{\varepsilon}

\renewcommand{\th}{\theta}

\renewcommand{\k}{\kappa}
\renewcommand{\l}{\lambda}

\newcommand{\s}{\sigma}
\newcommand{\vs}{\varsigma}

\newcommand{\om}{\omega}

\newcommand{\gTh}{\Theta}

\newcommand{\Om}{\Omega}

\newcommand{\ds}{\: \mathrm{d}s}

\newcommand{\du}{\: \mathrm{d}u}
\newcommand{\dx}{\: \mathrm{d}x}

\newcommand{\dds}{\mathrm{d}s}

\newcommand{\ddu}{\mathrm{d}u}
\newcommand{\ddx}{\mathrm{d}x}
\newcommand{\ddy}{\mathrm{d}y}

\newcommand{\la}{\langle}
\newcommand{\ra}{\rangle}
\newcommand{\bsl}{\backslash}
\newcommand{\pinf}{+ \infty}
\newcommand{\eqlaw}{\overset{(d)}{=}}
\newcommand{\convlaw}{\overset{(d)}{\to}}
\newcommand{\emp}{\emptyset}
\newcommand{\un}{\mathds{1}}
\newcommand{\unn}[1]{\mathds{1}_{\left \{ #1 \right \} }}
\newcommand{\Id}{\mathrm{Id}}

\newcommand{\sh}{\mathop{\mathrm{sh}}}

\renewcommand{\tanh}{\mathop{\mathrm{th}}}

\hypersetup{colorlinks,citecolor=green,filecolor=purple,linkcolor=red,urlcolor=blue}
\hypersetup{bookmarksdepth = 0}

\newtheorem{thm}{Theorem}[section]
\newtheorem{prop}[thm]{Proposition}
\newtheorem{lemma}[thm]{Lemma}

\theoremstyle{definition}
\newtheorem{defn}[thm]{Definition}

\theoremstyle{remark}
\newtheorem{remark}[thm]{Remark}

\makeatletter
\def\maketitle{%
  \null%
  \renewcommand{\thefootnote}{\fnsymbol{footnote}}
  \begin{center}\leavevmode
    \normalfont
    {\LARGE \@title\par}%
    \vskip 0.6cm
    {\large \@author\footnote{\@location. \par E-mail address: \@email}\par}%
  \end{center}%
  \null
	\renewcommand{\thefootnote}{\arabic{footnote}}
	\setcounter{footnote}{0}
  }
\def\location#1{\def\@location{#1}}
\def\email#1{\def\@email{#1}}
\makeatother

\setmarginsrb{2cm}{2cm}{2cm}{3cm}{0cm}{0cm}{0cm}{1cm}

\newcommand{\vsinf}{\vs_{\infty}}

\newcommand{\f}{\mathrm{f}}

\newcommand{\df}{\mathrm{d}\f}

\newcommand{\piN}{\pi^{(N)}}
\newcommand{\gN}{\g^{(N)}}

\newcommand{\YN}{Y^{(N)}}
\newcommand{\XNf}{X_{\f}^{(N)}}
\newcommand{\Xf}{X_{\f}}
\newcommand{\rN}{r^{(N)}}
\newcommand{\cZN}{\cZ^{(N)}}
\newcommand{\cZNr}{\cZ^{(N),\downarrow}}
\newcommand{\FN}{\F^{(N)}}
\newcommand{\FNr}{\F^{(N),\downarrow}}
\newcommand{\nuN}{\nu^{(N)}}

\newcommand{\tpi}{\widetilde{\pi}}

\newcommand{\rma}{\mathrm{a}}
\newcommand{\rmaN}{\mathrm{a}^{(N)}}

\newcommand{\trho}{\tilde{\rho}}
\newcommand{\SN}{S^{(N)}}
\newcommand{\LN}{L^{(N)}}
\newcommand{\bfe}{\mathbf{e}}

\newcommand{\n}{\mathbf{n}}
\newcommand{\m}{\mathbf{m}}

\newcommand{\PN}{\cP^{(N)}}
\newcommand{\Tmigr}{T_{\mathrm{migr}}}
\newcommand{\tc}{\tilde{c}}
\newcommand{\cMinf}{\cM^+_{\infty}}
\newcommand{\tpiN}{\widetilde{\pi}^{(N)}}
\newcommand{\tgN}{\widetilde{\g}^{(N)}}

\title{Two population models with constrained migrations}
\author{Raoul \textsc{Normand}}
\location{Laboratoire de Probabilit\'es et Mod\`eles Al\'eatoires - Universit\'e Paris VI \& Department of Mathematics - University of Toronto}
\email{raoul.normand@upmc.fr}

\begin{document}

\maketitle

\begin{abstract}
We study two models of population with migration. We assume that we are given infinitely many islands with the same number $r$ of resources, each individual consuming one unit of resources. On an island lives an individual whose genealogy is given by a critical Galton-Watson tree. If all the resources are consumed, any newborn child has to migrate to find new resources. In this sense, the migrations are constrained, not random. We will consider first a model where resources do not regrow, so the $r$ first born individuals remain on their home island, whereas their children migrate. In the second model, we assume that resources regrow, so only $r$ people can live on an island at the same time, the supernumerary ones being forced to migrate. In both cases, we are interested in how the population spreads on the islands, when the number of initial individuals and available resources tend to infinity. This mainly relies on computing asymptotics for critical random walks and functionals of the Brownian motion.
\end{abstract}

\begin{scriptsize}

\end{scriptsize}

\section{Introduction}

The incentive for this work is the series of three papers by Bertoin (see \cite{BertoinToA,BertoinAP,BertoinAR} and references therein) considering population models with neutral mutations, the latter occurring randomly. Mutations can also be viewed as migrations from an island to another, the individuals living on the same island being exactly those with the same alleles. While considering random mutations is natural, it is just as legitimate to assume that migrations do not happen randomly but are constrained: individuals will migrate when they need to find new resources to survive.

We shall study two models, which can be loosely described as follows. Individuals live on different islands, and each consumes one unit of resources to live. If one is born on an island where the resources have run out, they will migrate to a virgin island (that is, where no one has ever lived) and found their own colony. We shall thus assume that there are infinitely many virgin islands, and moreover, that each contains the same quantity of resources, denoted $r$ in the following. Our two models differ only by one point: in the first one, the resources do not regrow, whereas they do in the second. In other words, in the first model, when $r$ individuals have lived on an island, the next ones being born on this same island will migrate. In the second model, when $r$ individuals live on an island at the same time, and more individuals are born at that time, the latter have to migrate. Our models are reminiscent of the virgin island model of Hutzenthaler \cite{Hutzenthaler}, though once again, in our case, migrations are not random.

Our goal is to study how the population spreads on the different islands, that is, compute the number of islands where $k$ people live, for $k \in \N$. We do not wish to give a dynamical version of this, and we will thus wait until the population is extinct, and compute the number $Z_k$ of islands where $k$ people have lived. To this end, we shall study the measure
\[
\sum_{k \in \N} Z_k \dl_k.
\]
This is thus probably a good time to be more precise and introduce the quantities we will consider.

For both models, our study shall be fourfold. The first step is to construct precisely the object of interest. This will be called the tree of isles, in reference to the tree of alleles of \cite{BertoinToA}, and it will encode all the relevant information. This is a multitype tree, and when referring to it, we shall adopt an unusual language. A vertex of this tree corresponds to an occupied island and we will thus call its vertices ``islands''. A descendant of an island $\cI$ is an island founded by a migrant coming from $\cI$, so we call ``colonies'' these descendants. Finally, the type of an island will be the total number of individuals who lived on it, and we shall thus instead refer to this type as its ``population''.

The construction of the tree of isles is deterministic from a given finite tree. Consider the latter as the genealogical tree of the first existing individual, whom we shall impiously call Lucy\footnote{Which is the name given to the earliest known hominin, before the subsequent discovery of the earlier remains of ``Ardi''. Another reason for this choice is that lately, the literature has grown quite fond of the name ``Eve cluster'', designating thus a group of individuals, not a single one.}. Her children may have to migrate to find new resources -- obviously, this depends on the model. They move each to a different island and found their own colony, from which other individuals may migrate, and so on. Then, the root of the tree of isles corresponds to the home island of Lucy, its colonies to the islands founded by the migrants from this island, and so on. We conclude the construction of the tree of isles by attaching to each island its population.

The second step of our study is to find a relevant way to compute the population and number of colonies of the home island of Lucy, i.e. of the root of the tree of isles. We will be able to encode this quantities through an exploration process of the genealogical tree of Lucy, and for each model, we shall give an algorithm to construct this process in such a way that this information can be easily read on it. For the first model, the construction is simple since it is just the usual breadth-first search algorithm. Unfortunately, for the second model, we are led to writing quite abstruse definitions, but the figures should be enough to enlighten the reader.

The third step is now to add some randomness, in taking the genealogical tree of Lucy to be a Galton-Watson tree. Let us fix once and for all a reproduction law $\rho$, which is critical, i.e. has mean 1, and a second moment $0 < \s^2 < \pinf$. In particular, this tree is a.s. finite, and we can construct the tree of isles  $\om$-wise. It will be clear that the branching property of the initial tree carries on to to the tree of isles, which turns out to be a multitype Galton-Watson tree. This explains in particular why, in the previous paragraph, we were only interested in encoding the population and number of colonies of the root of the tree of isles: once we know this is a multitype Galton-Watson tree, these quantities are indeed enough to characterize it.

This ends what we can do starting from one individual and a fixed number of resources. Our last step is to pass to the limit, in different senses which will be made precise later. To this end, we will start from a number $N$, meant to go to infinity, of initial individuals spread on $N$ different islands. Their genealogies are assumed to be given by i.i.d. Galton-Watson trees with reproduction law $\rho$. We wish to rescale the number of resources so that, on every island, there is a probability of order $1/N$ that some people migrate. This is similar to the rate of mutation considered in \cite{BertoinToA}, or what is classically done in the Wright-Fisher model when considering rare mutations. In this case, there shall thus be a Poisson number of these initial islands which will have colonies. Clearly, this implies that there should be a number $r = r_N$ of resources tending to infinity, but the precise speed actually depends on the model. The main part of the work will be to compute the empirical measure describing how the population is spread on the different islands, more precisely its limit  after a proper rescaling. We will also provide a direct way to construct this limiting measure, which enlightens the structure of the tree of isles.

Finally, for the first model, we will provide an extra result, concerning the limit of the tree of isles in terms of trees. In other words, the goal here is to keep track of the genealogy of the islands, which is lost in the mere computation of the empirical measure. A similar result could be obtained for the second model, but it would just be a technical modification of the first one, and would not bring, we think, more understanding of the model.

The paper is organized in four sections following this introduction. The second is devoted to defining the objects we will consider and explain the techniques we shall use. In the third, we will recall some definitions about exploration processes, most of it being rephrasing of known (but quite sparse in the literature) folklore. The fourth and fifth section then deal each with one model. The second model is more involved than the first one and we can thus only advise the reader to follow the given order. We however hope to show that the second model has an interest on its own, on the one hand because it is probably more natural to consider regrowing resources, and on the other hand because it leads to the computation of nice asymptotics for random walks and formulas for the Brownian motion.

\paragraph{Acknowledgments} The author would like to thank Jean Bertoin for suggesting to study this topic, and for his always accurate insights. Heartily thanks also to Lorenzo Zambotti, Amaury Lambert and Olivier H\'enard for stimulating discussions, as well as to Marc Yor for providing several less known results on Bessel processes.

\section{Definitions and techniques}

\subsection{Trees}

We shall always consider rooted and ordered trees, thus having natural notions of ancestors, descendants, or individuals being on the right or the left of another. We shall adopt the following classical formalism. First, define the universal tree
\[
\cU = \bigcup_{n \in \Z^+} \N^n
\]
with $\N^0=\{\emptyset\}$, and let $\cU^*=\cU \bsl \{\emp\}$. If $u \in \N^k$, we say that $u$ is at generation $k$ and write $|u|=k$. The root is $\emp$, and the children of an individual $u=(u_1,\dots,u_k)$ at generation $k$ are $u\,j:=(u_1,\dots,u_k,j)$ for $j \in \N$. We call a tree rooted at $\emp$ a subset $\cT$ of $\cU$ such that
\begin{itemize}
	\item $\emp \in \cT$;
	\item if $v \in \cT$ and $v = u \, j$ for some $u \in \cU$ and $j \in \N$, then $u \in \cU$;
	\item for every $u \in \cT$, there exists a number $k_u(\cT) \in \Z^+$ such that $u \, j \in \cT$ if and only if $j \leq k_u(\cT)$.	
\end{itemize}
The quantity $k_u(\cT)$ is the number of children, or descendants, of $u$ ($k$ stands for ``kids''). A tree rooted at $u \in \cU$ is a subset $\cT$ of $\cU$ which can be written as $u \, \cT'$, where $\cT'$ is a tree rooted at $\emp$. If no detail is given, a tree will always be assumed to be rooted at $\emp$, though we will just call tree a tree rooted at another vertex if the context is clear. A Galton-Watson tree (with types or not) can thus be defined in a natural way, see \cite{DuquesneLeGall,LegallMBPB}. We let $s(\cT)$ to be the size of $\cT$, that is, its cardinality. A labeling of a tree $\cT$ is a bijection between $\cT$ and $\{1, \dots, s(\cT) \}$, and, given a labeling, we may refer to ``vertex $i$'' instead of ``the vertex labeled $i$''.

We call these trees discrete and will also use, for the second model, continuous trees, that is trees whose branches have lengths, interpreted as life-lengths of the individuals: the latter give birth to all of their children just before dying. Let us not dwell on an obscure formal definition, since once again, figures shall make things lucid.

A multitype tree $\cA$ is a tree such that a type $\cA_u$ is attached to each vertex $u$ of the tree. The types here will be positive, and we will give type 0 to the vertices which do not belong to $\cA$, i.e. define $\cA_u = 0$ for $u \in \cU \bsl \cA$. In particular, the set of vertices with non-zero type is a tree, and we will thus say that $u \in \cA$ if $\cA_u \neq 0$.

\subsection{Construction of the tree of isles} \label{sec:treeofisles}

Let us start by recalling that the trees involved in the first model will be discrete, whereas they shall be continuous for the second model. Let us fix a tree $T$, interpreted s the genealogical tree of an individual funding an island $\cI$. For each model, we will define (respectively in Section \ref{sec:treeofisles1} and \ref{sec:treeofisles2}) the migrant children of a tree $T$, which is just a subset of vertices of $T$. We denote $C(T)$ their number, which is, in our terminology, the number of colonies of $\cI$, though we may also refer to it as the number of colonies of $T$. These individuals are roots of disjoint\footnote{This will be clear from the precise definition of these migrant children.} trees which we will denote $\Tmigr^1, \dots, \Tmigr^{C(T)}$. After we cut off these subtrees from $T$, we obtain a pruned tree, which describes the genealogy of this individual restricted to its offspring living on $\cI$. The size of this pruned tree shall be denoted $P(T)$, and is the population of $\cI$, or of $T$. We start with the genealogical tree $\cT$ of Lucy. 

This notation obviously depend on the model, but we shall not specify it: clearly, it refers to the first model in Section \ref{sec:firstmodel} and to the second one in Section \ref{sec:secondmodel}. However, it depends on the number of resources we consider, and we shall add an index $r$ if necessary, writing $P_r(\cT)$ and $C_r(\cT)$ to specify that we deal with $r$ resources.

The tree of isles of $\cT$, $\cA$, can then be constructed recursively as follows. Recall that it is a multitype tree, whose vertices are called ``islands'', types ``population'' and descendants ``colonies''.
\begin{itemize}
	\item First, take $X = \{ (\emp,\cT) \}$.
	\item At each step, if $X = \emp$ then stop. Otherwise, pick some $(v,T)$ in $X$, and remove it from $X$. Add island $v$ to $\cA$, and give it population $P(T)$ and $C(T)$ colonies (if any). Then, for $i \in \{ 1, \dots, C(T) \}$, add $(v \, i, \Tmigr^i)$ to $X$, and go to the next step.
\end{itemize}
It should be clear that this algorithm provides a multitype tree, which is, by definition, the tree of isles constructed from $\cT$.

It is already worth keeping in mind the following properties of the tree of isles, which shall be proven later:
\begin{itemize}
	\item its construction is deterministic;
	\item the information which is relevant to our study is encoded in the tree of isles;
	\item if the initial genealogical tree is a Galton-Watson tree, then the tree of isles is a multitype Galton-Watson tree.
\end{itemize}

Let us finally take advantage of this section to fix some notation. A deterministic tree shall be denoted $\cT$, and its tree of isles $\cA$ (from French ``arbre'', tree). A random tree (with no type, and it will always be a Galton-Watson tree) will be denoted $\T$, and the corresponding tree of isles $\A$. Note once again that the latter is just obtained deterministically from $\T$: each $\T(\om)$ gives rise to a $\A(\om)$.

\subsection{Notation and empirical measure} \label{sec:empmeas}

Let us first introduce some notation.
\begin{itemize}
	\item $C_K$ is the space of nonnegative continuous functions with compact support in $(0, \pinf)$;
	\item $\cM^+$ is the space of nonnegative Radon measures on $(0,\pinf)$, endowed with the vague topology (i.e. the space of test functions is $C_K$);
	\item we denote $\Rightarrow$ the convergence in $\cM^+$, and $\convlaw$ the usual convergence in distribution of real random variables;
	\item for $\cP \in \cM^+$ and $f \in C_K$, $\la \cP, f \ra$ is the integral of $f$ with respect to $\cP$;
	\item finally, for notational simplicity, in the whole document, we do not write the integer parts.
\end{itemize}

As mentioned in the introduction, we want to start with a large number of Galton-Watson trees with reproduction law $\rho$, which has mean 1 and a second moment $0 < \s^2 < \pinf$. We will thus pick $(\Om, \cF, \P)$ a probability space where an i.i.d. family $(\T^k)_{k \geq 0}$ of those trees can be defined. We can then construct, for each model, the corresponding trees of isles $(\A^k(r))_{k \geq 0}$ corresponding to $r$ resources, which are also i.i.d. since the construction of the tree of isles is deterministic. 

Remember that we wish to have the number $N$ of initial islands, as well as the number of resources $r_N$ on each island, tend to infinity, and to study the empirical measure of the population. We thus define, for each $r,n \in \N$,
\[
\cP_n(r) = \sum_{k=1}^n \sum_{u \in \cU} \dl_{\A_u^k(r)} \unn{\A_u^k(r) \neq 0}.
\]
In other words, the mass of $\cP_n(r)$ at $l \in \N$ is the number of islands where $l$ people have lived, when starting from $n$ islands and considering $r$ resources. We are interested in the convergence of $\cP_N(r_N)$, but to obtain a nontrivial limit, we will need to rescale it, by setting
\[
\PN_n(r) = \sum_{k=1}^n \sum_{u \in \cU} \dl_{\A_u^k(r)/N^2} \unn{\A_u^k(r) \neq 0}.
\]
This rescaling actually implies (see Lemma \ref{lem:tightness}) the tightness of $(\PN_N(r_N))$ in $\cM^+$. Clearly, $\PN_n(r)$ is just the sum of $n$ independent copies of the reference measure $\PN(r) := \PN_1(r)$. We also define $\T := \T^1$ and $\A(r) := \A^1(r)$, so $(\T^k)$ (resp. $(\A^k(r))$) is a family of i.i.d. copies of $\T$ (resp. $\A(r)$). Finally, we let, for $p \geq 1$, $i \geq 0$,
\[
Z_{i,p}(r) = \# \{ u \in \cU, \, |u| = i, \, \A_u(r) = p \},
\]
so we can rewrite
\[
\PN(r) = \sum_{p \geq 1} \dl_{p/N^2} \sum_{i \geq 0} Z_{i,p}(r)
\]
what will turn out to be very useful to write an equation solved by the cumulant of $\PN_N(r_N)$: recall that the cumulant of a random measure $\cP$ on $\cM^+$ is
\[
\k(f) = - \ln \E \left ( \exp - \la \cP , f \ra \right )
\]
for $f \in C_K$, and that the knowledge of $\k(f)$ for every $f \in C_K$ characterizes $\cP$.

\subsection{General technique of proof} \label{sec:technique}

To compute the limit of $(\PN_N(r_N))$, we will proceed in several steps. The first is to prove the tightness, which is easy. The second is to prove the uniqueness of the limit. To this end, we first give an equation solved by the cumulant $\k_N(f)$ of $\PN_N(r_N)$, which is obtained thanks to the branching property of the tree of isles. We then pass to the limit in this equation to obtain an equation solved by any subsequential limit of $\k_N(f)$, which is readily proved to have a unique solution.  This guarantees, with the tightness, the convergence of $(\PN_N(r_N))$.

The main part of this work is to compute the limit of the unknown quantity involved in the aforementioned equation, namely the population and number of colonies of the root of the tree of isles. As we mentioned, these quantities can be read on a random walk (which is the exploration process of the genealogical tree of Lucy), and our work will then boil down to finding limits of functionals of random walks. This is undoubtedly the most technical part of the work; however, if we leave aside the necessary computations, most of the results should not come as a surprise.

\section{Exploration processes}

In this section, we shall provide a general way to build the exploration process of a tree. The goal of this construction is that, if this tree is a Galton-Watson tree, then its exploration process is a random walk. These matters are classical, and we shall not dwell on the proofs.

\subsection{A general construction of the exploration process} \label{sec:explproc}

Consider a finite tree $\cT$, with size $s = s(\cT)$. There are several ways to label it, i.e. write a bijection between the set of vertices and $\{1,\dots,s\}$, the most common being by breadth-first and depth-first search. From a labeling, we may construct the exploration process as follows: denote $k_i$ the number of children of (the individual labeled) $i$. Then the exploration process is given by $S_0 = 0$ and
\[
S_i = (k_1 - 1) + \dots + (k_i - 1), \quad i = 1 \dots s,
\]
and we let $S_i = -1$ for $i > s$ for the sake of definiteness. We will extend this process to the whole of $\R^+$ by interpolating linearly for the first model, whereas we will instead set $S_t = S_{\lfloor t \rfloor}$ for the second model. When $\cT$ is a Galton-Watson tree with reproduction law $\rho$, and when it is labeled by breadth-first or depth-first search, it is well-known that $(S_n)$ is a left-continuous random walk, with step distribution $\trho$, absorbed at $-1$ at time $s$, where $\trho(k) = \rho(k+1)$ for $k \geq -1$.

Let us now explain a more general way to label $\cT$ so as to conserve this last property. Let $\cT$ be a discrete or continuous tree. We say that $\ell$ is a \emph{line} in $\cT$ if $\ell$ is a set of vertices such that every path from the root to a leaf contains at most one vertex of $\ell$. We now define $\cT_{\ell}$ the tree pruned at level $\ell$, as follows.
\begin{itemize}
	\item In the discrete case, $\cT_{\ell}$ is the connected component of the root when we delete the descendants of the individuals of $\ell$: see Figure \ref{fig:pruneddisctree}.
	\item In the continuous case, $\cT_{\ell}$ is the connected component of the root when we delete the descendants of the individuals of $\ell$, but keep the life-lengths of all the individuals in $\ell$: see Figure \ref{fig:prunedconttree}.
\end{itemize}

\begin{figure}[htb]
\centering
\includegraphics[width=0.75 \columnwidth]{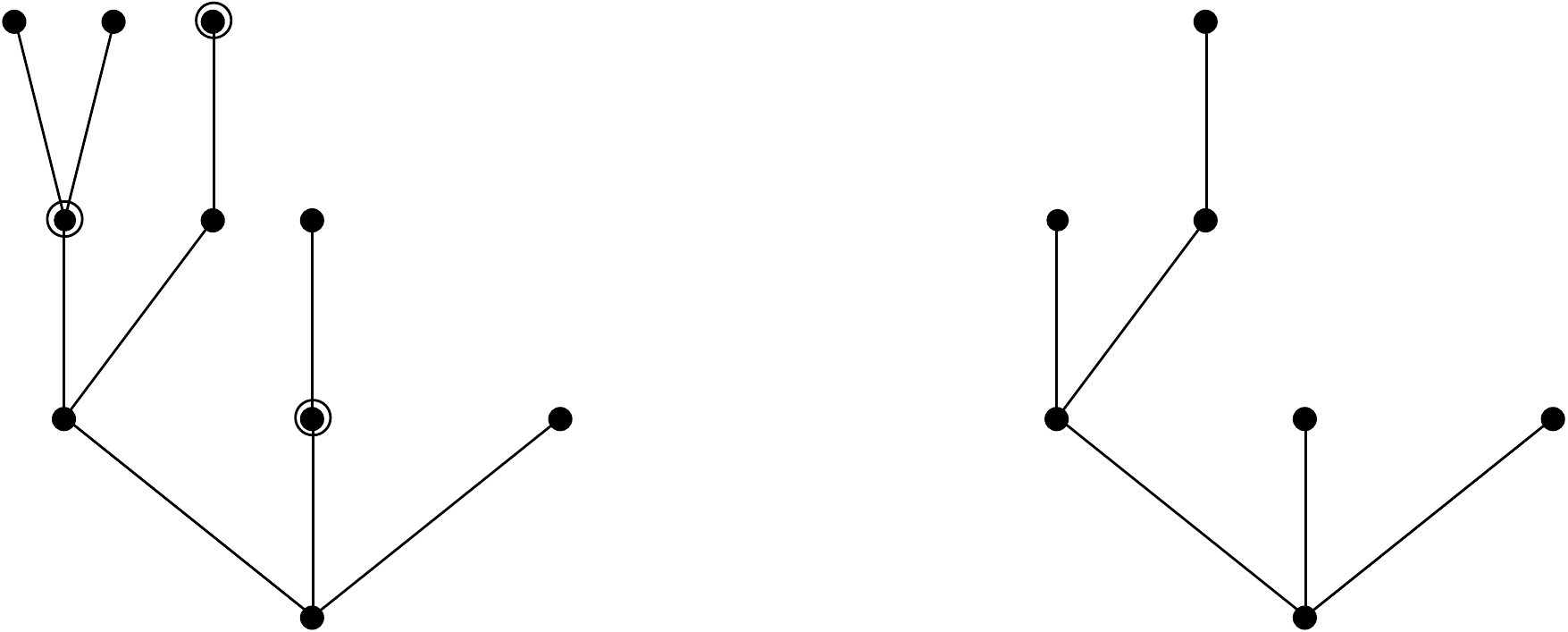}
\caption{A discrete tree $\cT$, a line $\ell$ (the circled vertices), and the tree $\cT_{\ell}$ pruned along this line.}
\label{fig:pruneddisctree}
\end{figure}

\begin{figure}[htb]
\centering
\includegraphics[width=0.75 \columnwidth]{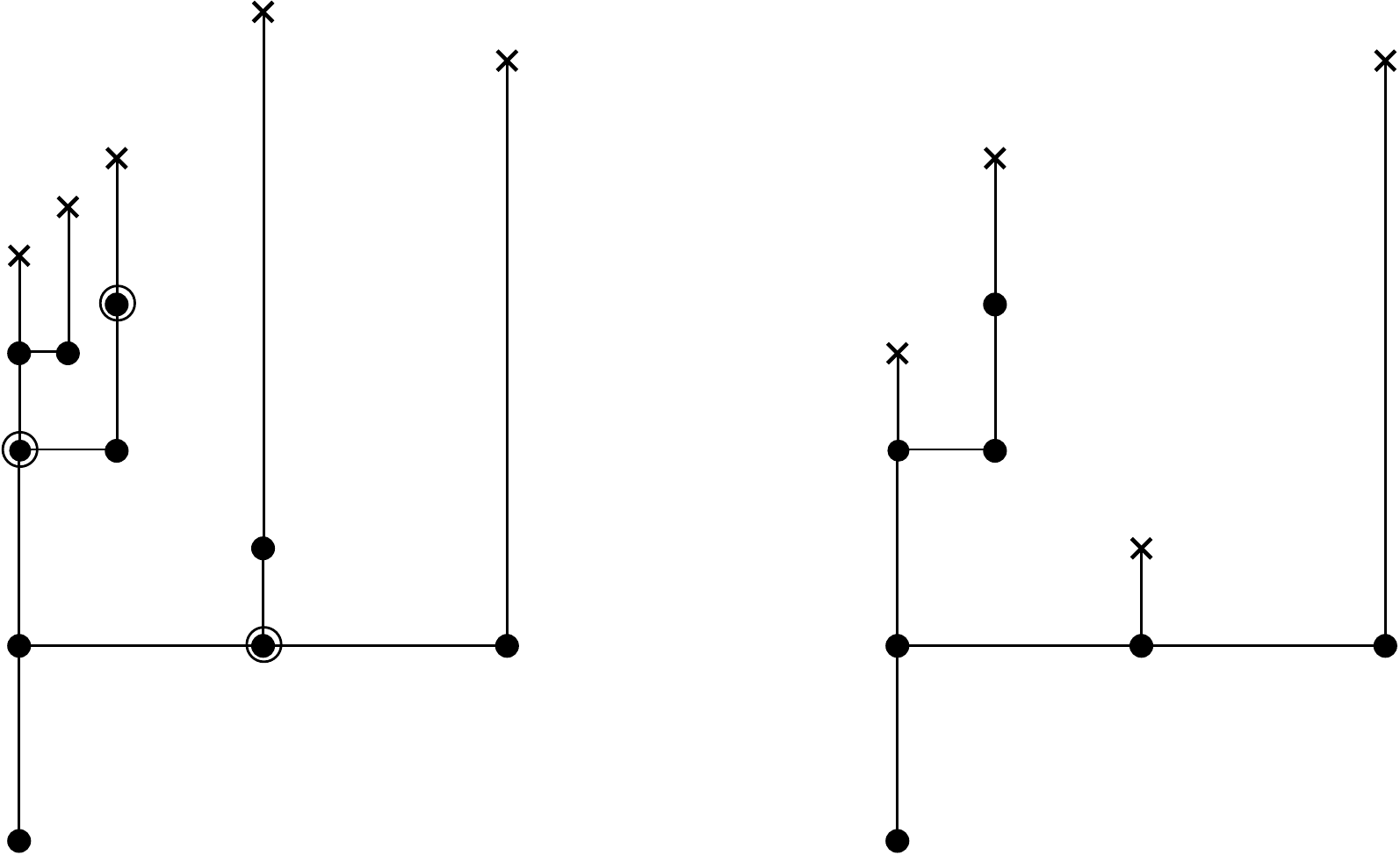}
\caption{A continuous tree $\cT$, a line $\ell$ (the circled vertices), and the tree $\cT_{\ell}$ pruned along this line. Note that we cut the descendants of the individuals in $\ell$, not the individuals in $\ell$, so the right-most individuals still remain in $\cT_{\ell}$.}
\label{fig:prunedconttree}
\end{figure}

We will label our tree $\cT$ by giving labels $1$ to $s(\cT)$ in this order, the label 1 going to the root. To explain which vertex we label $i+1$ after we have given labels $1$ to $i$, we use what we call a \textbf{Markovian rule}.

\begin{defn} \label{def:rule}
Let $\cL$ be the set of nonempty lines in $\cT$. A Markovian rule $\cR$ is a mapping from $\cL$ to the set of vertices of $\cT$ such that, for every $\ell \in \cL$, $\cR(\ell) \in \ell$ and $\cR(\ell)$ depends only on $\cT_{\ell}$.
\end{defn}

The labeling of $\cT$ corresponding to this rule is then given by the following algorithm.

\begin{enumerate}
  \item At step 0, take $\ell$ to be the root of $\cT$, and go to step $1$.
 	\item At step $i$:
 	\begin{itemize}
	\item if $\ell = \emp$, then stop;
	\item if not, choose the vertex $v$ according to the rule $\cR$, i.e. take $v = \cR(\ell)$. Give $v$ the label $i$. Then remove $v$ from $\ell$, and add to $\ell$ the children of $v$;
	\item then go to step $i+1$.
	\end{itemize}
\end{enumerate}
It should be clear that when the algorithm stops, the tree is labeled. Let $(S_n)$ to be the exploration process corresponding to this labeling. We claim the following.

\begin{lemma} \label{lem:explproc}
\begin{itemize}
\item The size $\s(\cT)$ of the tree is the hitting time of $-1$ by $(S_n)$.
\item If $\cT = \T$ is a (discrete or continuous) Galton-Watson tree with reproduction law $\rho$, then, at each step of the labeling algorithm, conditionally on $\ell$, the subtrees rooted at the vertices of $\ell$ are i.i.d. with the same law as $\T$. In particular, $(S_n)$ has the law of a random walk with step distribution $\trho$, absorbed when first hitting $-1$.
\end{itemize}
\end{lemma}

\begin{proof}
The first part is classical and easily proven by induction. The second part is just the observation that, at each step of the labeling algorithm, the fact that we choose a Markovian rule ensures that the set of edges above $\ell$ is a stopping line, in the terminology of \cite{Chauvin} (see also the nice informal explanation in \cite{BertoinToA}). This ensures that the branching property holds for the subtrees rooted at $\ell$, i.e. that conditionally on $\ell$, the subtrees rooted at the vertices of $\ell$ are i.i.d. with the same law as $\T$. This clearly implies the statement about the exploration process.
\end{proof}

This result obviously encompasses the two known cases mentioned. The rule in the breadth-first search case is ``pick the leftmost vertex at the lowest generation in $\ell$'', and ``pick the smallest vertex in the lexicographical order in $\ell$'' in the depth-first search case. We may make up a lot of valid rules, and this construction will be mostly useful when we study the second model, since the rule then is quite involved.

\subsection{Death-first search algorithm} \label{sec:deathfirst}

Let us now introduce a particular labeling of a continuous tree by what we shall call ``death-first search''. Consider a continuous tree $\cT$, such that two events (birth or death) do not occur at the same time. Our rule is, for a line $\ell$, to pick the individual in $\ell$ who dies first. Since $\cT_{\ell}$ keeps track of the life-lengths of the individuals in $\ell$, this is clearly a Markovian rule. See Figure \ref{fig:death_first} for an example, and Figure \ref{fig:death_first_expl_proc} for the corresponding exploration process $(S_n)$.

\begin{figure}[htb]
\centering
\includegraphics[width=0.5 \columnwidth]{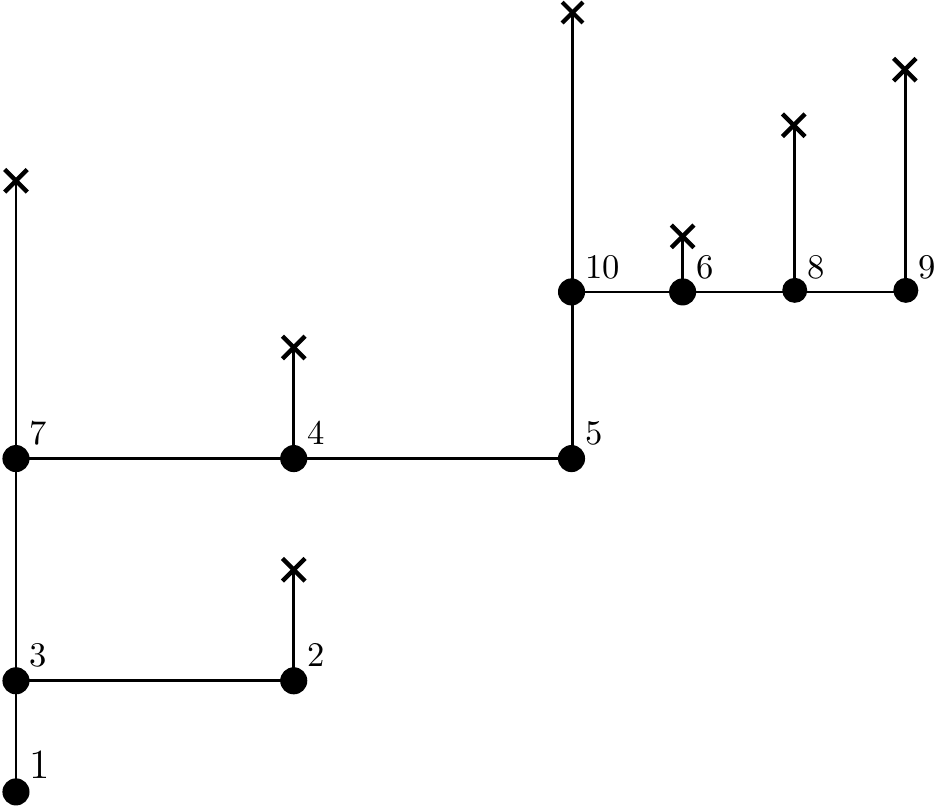}
\caption{Labeling of a tree by death-first search.}
\label{fig:death_first}
\end{figure}

\begin{figure}[htb]
\centering
\includegraphics[width=0.5 \columnwidth]{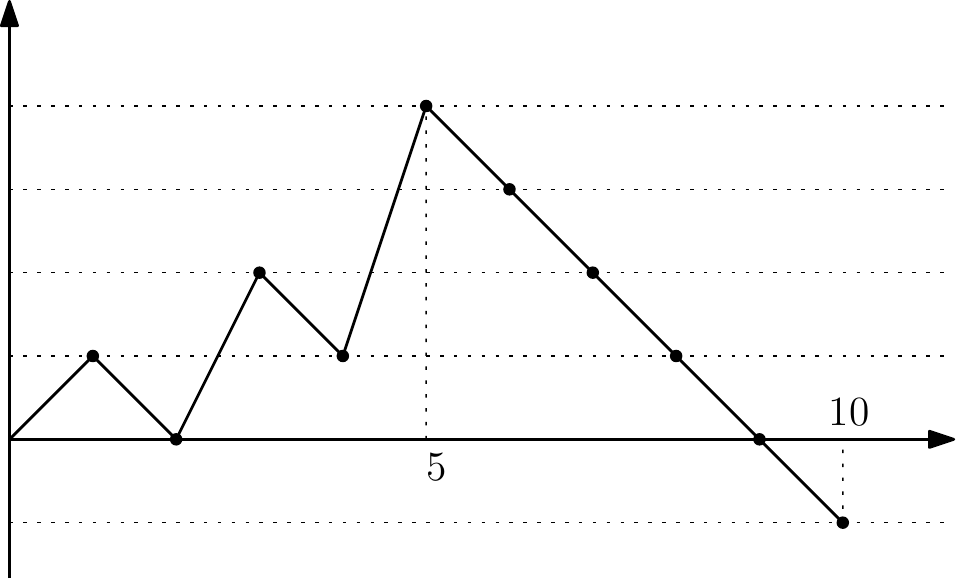}
\caption{Exploration process corresponding to the labeling of the tree in Figure \ref{fig:death_first}.}
\label{fig:death_first_expl_proc}
\end{figure}

Now, let $0 = \tau_0 < \tau_1 < \dots < \tau_k < \tau_{k+1} = \pinf$ be the times of the successive events of birth or death. Then $1 = 1+S_0$ is the number of individuals living on $[\tau_0,\tau_1)$. Then the number of children of 1 is $1 + S_1$, so the number of people alive on $[\tau_1,\tau_2)$ is $1 + S_1$. The number of children of 2 is $1 + (S_2 - S_1)$, so the number of individuals alive on $[\tau_2,\tau_3)$ is
\[
\underbrace{1 + S_1}_{\text{individuals alive at $\tau_2^-$}} + \overbrace{- 1}^{\text{death of 2}} + \underbrace{1 + S_2 - S_1}_{\text{number of children of 2}} = 1 + S_2.
\]
By a clear induction, one readily sees that $1+S_i$ is the number of individuals alive on $[\tau_i,\tau_{i+1})$, for $i \in \{ 0, \dots, k \}$. We will use a variation of this algorithm and this observation when studying our second model.

\section{First model: fossil resources} \label{sec:firstmodel}

Let us recall informally our first model, which involves fossil resources, i.e. non-regrowing resources. Each island $\cI$ contains a number $r \in \N$ of resources, which are consumed by the first $r$ individuals living on $\cI$. Every new individual born on $\cI$ after these $r$ first will migrate, each to a new virgin island with $r$ resources, to found its own colony.

\subsection{Tree of isles} \label{sec:treeofisles1}

Let us start by defining the tree of isles $\cA(r)$ from a discrete tree $\cT$. According to Section \ref{sec:treeofisles}, to this end, all we need to do is to explain how we choose its migrant children, which should be clear from the informal description of the model. Since several children may be born at the same time, the only issue is to choose which children migrate after the resources on the island have been exhausted, and we will pick arbitrarily the right-most children.

To make this precise, if $s(\cT) \leq r$, then $\cT$ has no migrant children. Else, label $\cT$ by breadth-first search. Then the migrant children are the individuals with label in $\{r+1, r+2, \dots \}$, which are descendants of individuals with label in $\{1, \dots, r \}$. See Figures \ref{fig:disc_tree} and \ref{fig:tree_of_isles_1} for examples.

\begin{figure}[htb]
\centering
\includegraphics[width=0.75 \columnwidth]{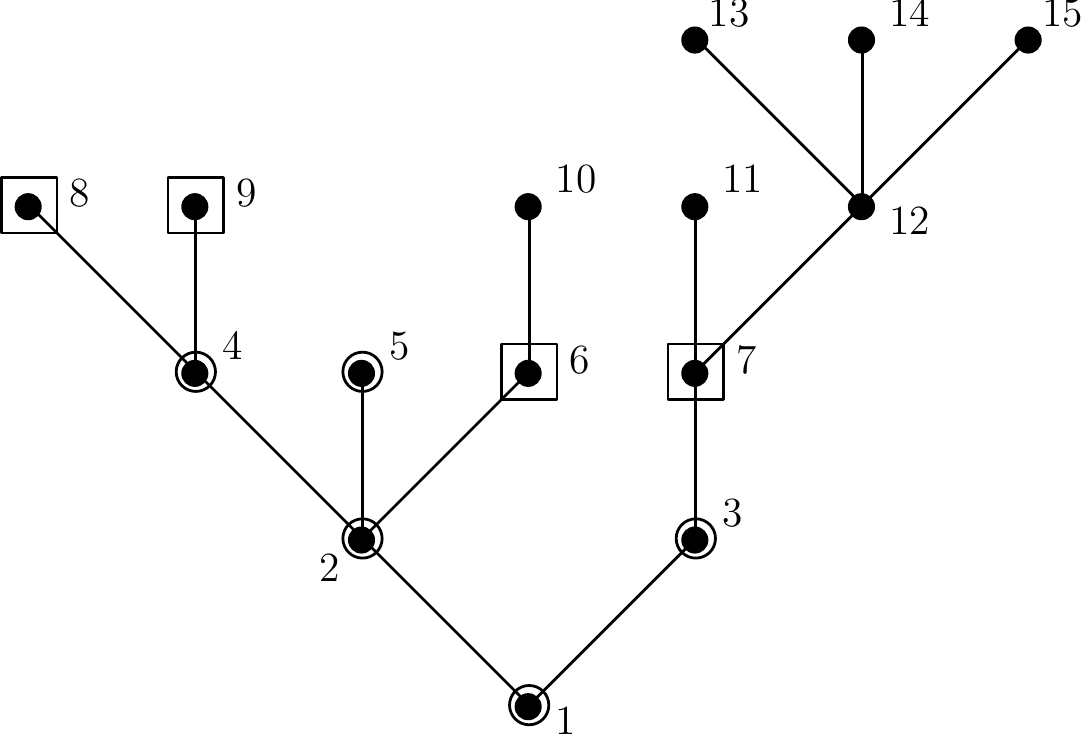}
\caption{A tree labeled $\cT$ by breadth-first search. Its migrants for $r=5$ are shown in a square, the individuals in a circle remaining on the island. Here $P_r(\cT)=5$ and $C_r(\cT) = 4$.}
\label{fig:disc_tree}
\end{figure}

\begin{figure}[htb]
\centering
\includegraphics[width=0.5 \columnwidth]{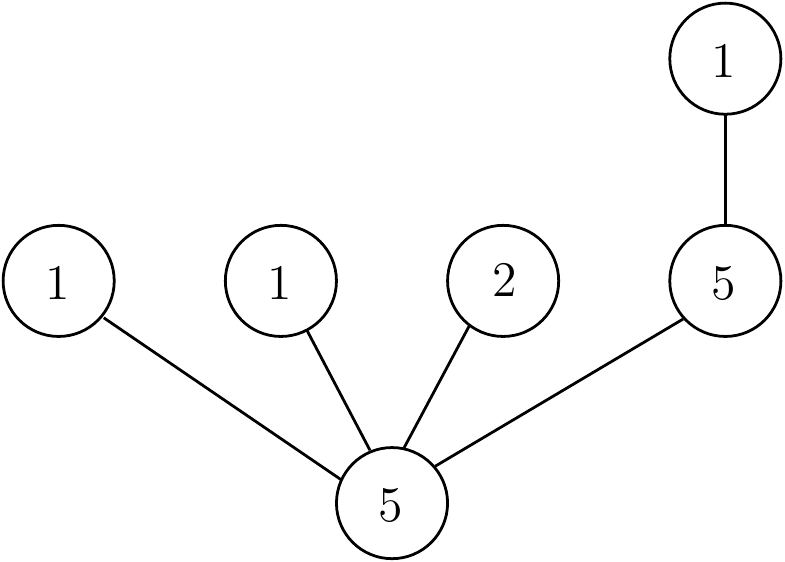}
\caption{The tree of isles $\cA(r)$ obtained from the tree $\cT$ of Figure \ref{fig:disc_tree}, for $r=5$.}
\label{fig:tree_of_isles_1}
\end{figure}

Let us go on with our plan: we wish to read the population and number of colonies of the root of $\cA(r)$ on the exploration process of $\cT$. To this end, we shall naturally consider $(S_n)$ the exploration process associated to the labeling by breadth-first search, see Figure \ref{fig:disc_expl_proc} for an example. We let
\[
\vsinf = \inf \{t \geq 0, \, S_i = - 1 \}, \quad \inf \emp := \pinf,
\]
the hitting time of $-1$ by $(S_n)$. We now claim the following.

\begin{lemma} \label{lem:popcolRW1}
The following equalities hold:
\[
P_r(\cT) = \vsinf \wedge r, \quad C_r(\cT) = 1 + S_{\vsinf \wedge r}.
\]
\end{lemma}

\begin{proof}
Recall from Lemma \ref{lem:explproc} that $\vsinf = s(\cT)$. Hence, when there are no migrants, i.e. $s(\cT) = \vsinf \leq r$, then $P_r(\cT) = s(\cT) = \vsinf = \vsinf \wedge r$ and $C_r(\cT) = 0 = 1 + S_{\vsinf} = 1 + S_{\vsinf \wedge r}$.

When there are migrants, i.e. $s(\cT) = \vsinf > r$, then by definition $P_r(\cT) = r = \vsinf \wedge r$. Now, by an easy induction as in Section \ref{sec:deathfirst}, it is easy to check that for each $i \in \{1, \dots, s(\cT)\}$, $1 + S_i$ is precisely the number of children of the individuals with label in $\{ 1,\dots,i \}$ whose label is in $\{ i+1, i+2, \dots \}$. Hence, by definition, $C_r(\cT) = 1 + S_r = 1 + S_{\vsinf \wedge r}$.
\end{proof}

\begin{figure}[htb]
\centering
\includegraphics[width=0.75 \columnwidth]{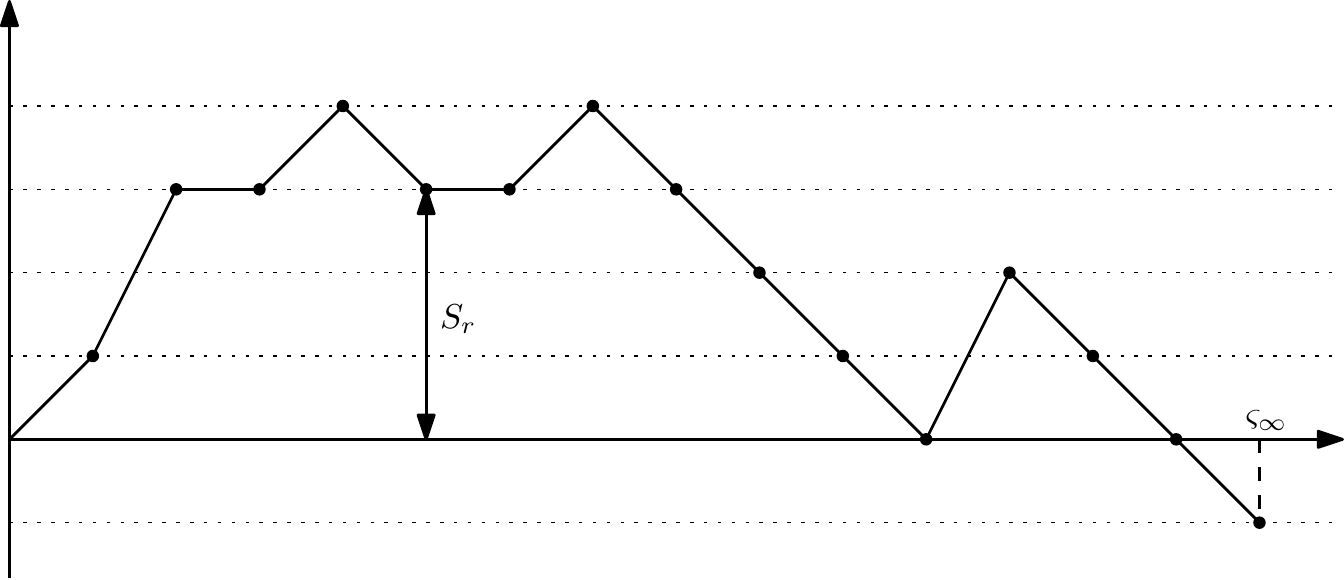}
\caption{The exploration process obtained from the tree $\cT$ of Figure \ref{fig:disc_tree}. Check that, for $r = 5$, $\vsinf \wedge r = 5 = P_r(\cT)$ and $1 + S_r = 4 = C_r(\cT)$.}
\label{fig:disc_expl_proc}
\end{figure}

Let us use the notation of Section \ref{sec:empmeas}, so $\T$ is a Galton-Watson tree with reproduction law $\rho$ and $\A(r)$ its tree of isles when we consider $r$ resources. Instead of considering the walk $(S_n)$ as being constructed from $\T$, we rather take $(S_n)$ to be an actual random walk with step distribution $\trho$, defined on all of $\Z^+$. We thus just have to replace the equalities in Lemma \ref{lem:popcolRW1} by equalities in law. Our objective now is to prove that $\A(r)$ is a multitype Galton-Watson tree. Taking this for granted, the only information we need to characterize it is the law of the type (i.e. population in our terminology) and number of children (or colonies) of the root, i.e. the law of the couple $(P_r(\T),C_r(\T))$. We let $\pi_r$ the law of $P_r(\T)$, and $\g_r$ the law of $C_r(\T)$ knowing that $P_r(\T) = r$. By definition, $C_r(\T) = 0$ when $P_r(\T) < r$, so the law of $(P_r(\T),C_r(\T))$ is indeed specified by $\pi_r$ and $\g_r$.

\begin{lemma} \label{lem:treeGW1}
The tree of isles $\A(r)$ is a $r$-type Galton-Watson tree, described as follows.
\begin{itemize}
	\item The population of the root is chosen according to $\pi_r$;
	\item the number of colonies of the islands of population $r$ has law $\g_r$, and each colony chooses independently its population according to $\pi_r$;
	\item the islands of population $1,\dots,r-1$ do not have any colonies.
\end{itemize}   
\end{lemma}

\begin{proof}
Let us take a look at the algorithm defining the tree of isles $\A(r)$. The root has, by definition, population given by $\pi_r$. If $\A_{\emp}(r) < r$, then $\T$ has no migrant children and $\A_u(r) = 0$ for $u \in \cU^*$. Else, $\A_{\emp}(r) = r$, and $\T$ has a number of migrant children given by $\g_r$. Conditioned on these migrant children, the second statement of Lemma \ref{lem:explproc} tells that the subtrees $\T_{\mathrm{migr}}^1,\dots,\T_{\mathrm{migr}}^{C_r(\T)}$ are independent with the same law as $\T$. A simple induction then yields the result.
\end{proof}

A straightforward corollary of this result concerns $(Z_{i,p}(r),\, p = 1, \dots, r)_{i \geq 0}$, as defined in Section \ref{sec:empmeas}.

\begin{lemma} \label{lem:procGW1}
The process $(Z_{i,p}(r),\, p = 1, \dots, r)_{i \geq 0}$ is a $r$-type Galton-Watson process, described as follows. 
\begin{itemize}
	\item The $r$-tuple $(Z_{0,1}(r),\dots,Z_{0,r}(r))$ has the same law as $(\unn{P_r(\T)=1},\dots,\unn{P_r(\T)=r})$;
	\item the number of children of the individuals of type $r$ has law $\g_r$, and each child chooses independently its type according to $\pi_r$;
	\item the individuals of type $1,\dots,r-1$ do not have any children.
\end{itemize} 
\end{lemma}

\subsection{Rescaling}

Let us keep on with our program. We now start from $N$ independent islands and $r_N$ resources, both meant to tend to infinity. We wish to rescale $r_N$ so that each island has a probability of order $1/N$ to have colonies. From Lemma \ref{lem:popcolRW1},
\begin{equation} \label{eq:tailvsinf}
\P(C_{r_N}(\T) > 0) = \P(\vsinf > r_N) \sim \sqrt{\frac{2}{\pi \s^2}} \frac{1}{\sqrt{r_N}}.
\end{equation}
The last equivalent stems from well-known facts about hitting times for random walks, see e.g. \cite{Spitzer}, p.382. We shall thus assume that, for some $c > 0$,
\begin{equation} \label{eq:hyprN1}
\lim_{N \to \pinf} \frac{r_N}{N^2} = c > 0.
\end{equation}

Let us introduce some notation. We denote for simplicity
\[
\l = \sqrt{\frac{2}{\pi \s^2 c}}.
\]
Let 
\[
P_N  = P_{r_N}(\T), \quad C_N = C_{r_N}(\T), \quad p_N = \P(P_N = r_N),
\]
their conditioned laws
\[
\piN = \cL(P_N | P_N < r_N), \quad \gN = \cL( C_N | P_N = r_N ),
\]
and their rescaled versions
\[
\tpiN = \cL \left ( \frac{P_N}{N^2} \middle | P_N < r_N \right ), \quad \gN = \cL \left ( \frac{C_N}{N} \middle | P_N = r_N  \right ).
\]
Let us first study the limit, in different senses, of the two latter. We define here
\[
\mu(\ddx) = \frac{1}{2} \frac{1}{x^{3/2}} \dx, \quad \mu^c(\ddx) = \unn{x \in (0,c)} \mu(\ddx)
\]
and
\[
\th = \cL(\sqrt{c} \s W_1^+), \quad\P(W_1^+ \in \ddx) = x e^{-x^2/2} \unn{x > 0} \dx.
\]
The notation $W_1^+$ just stems from the usual notation $(W^+_t)_{t \in [0,1]}$ for the standard Brownian meander, and $W_1^+$ is its tip, whose law is the Rayleigh law $x e^{-x^2/2} \unn{x > 0} \dx$.

\begin{lemma} \label{lem:convmeas}
As $N \to \pinf$,
\[
p_N \sim \frac{\l}{N}
\]
and moreover, the convergences
\[
N \tpiN \Rightarrow \l \mu^c, \quad \tgN \convlaw \th
\]
hold, respectively in $\cM^+$ and in law.
\end{lemma}

\begin{proof}
The first statement is just \eqref{eq:tailvsinf} written with a non-strict inequality. For the first convergence, note that more generally, \eqref{eq:tailvsinf} and \eqref{eq:hyprN1} ensure that, for every $a > 0$,
\[
N \P ( \vsinf > a N^2 ) \to \frac{\l}{\sqrt{a}} = \frac{\l}{2} \int_a^{\pinf} \frac{1}{x^{3/2}} \dx,
\]
so one readily sees by standard approximations that
\[
N \P \left ( \frac{\vsinf}{N^2} \in \ddx \right ) \Rightarrow \frac{\l}{2} \frac{1}{x^{3/2}} \dx.
\]
Recall also from Lemma \ref{lem:popcolRW1} that $P_N$ has the same law as $\vsinf \wedge r_N$, and thus, for $f \in C_K$,
\begin{align*}
& \int_0^{\pinf} f(x) N \tpiN(\ddx) \\
& = \P ( \P_N < r_N )^{-1} \int_0^{(r_N - 1)/N^2} f(x) N \P ( P_N / N^2 \in \ddx) \\
& = (1-p_N)^{-1}  \int_0^{(r_N - 1)/N^2} f(x) N \P ( \vsinf / N^2 \in \ddx) \\
& = (1-p_N)^{-1} \left ( \int_0^c f(x) N \P ( \vsinf / N^2 \in \ddx) + \int_c^{(r_N - 1) / N^2} f(x) N \P ( \vsinf / N^2 \in \ddx) \right ) \\
& \to \int_0^c f(x) \mu(\ddx)
\end{align*}
using the computation above and that the second term is easily seen to tend to 0 since $r_N/N^2 \to c$.

For the second convergence, note that Lemma \ref{lem:popcolRW1} implies that $\tgN$ is the law of
\[
\left. \frac{1 + S_{r_N}}{N} \right | \vsinf \geq r_N.
\]
It is well-known (see \cite{Bolt}) that a centered random walk with a second moment, conditioned to stay positive, converges to the Brownian meander $(W^+_t)_{t \in [0,1]}$, and in particular
\[
\left. \frac{S_{r_N}}{\s \sqrt{r_N}} \right | \vsinf \geq r_N \convlaw W^+_1,
\]
whence the result follows after noticing (see e.g. \cite{Igle}) that $W^+_1$ has the Rayleigh law.
\end{proof}

\subsection{Heuristics and result} \label{sec:heuristics}

Before stating the result, we will, with the help of the previous results, discuss some heuristics. Consider the forest
\[
\left ( \A^1(r_N),\dots,\A^N(r_N) \right ) / N^2,
\]
where dividing by $N^2$ means rescaling the population of every island by $1/N^2$. The islands with (rescaled) population $1/N^2, \dots, (r_N - 1) / N^2$ do not have colonies, whereas those with population $r_N / N^2 \approx c$ may\footnote{And probably do: an island of population $r_N/N^2$ does not have colonies only in the (rare) case where the founder of this island has precisely $r_N - 1$ descendants.}. We shall thus call the latter type of islands \emph{fertile}. According to Lemma \ref{lem:convmeas}, the number of colonies of a fertile island is approximately $\f N$, where $\f$ has law $\th$. We will then say that this island has fertility $\f$.

Now, by Lemma \ref{lem:treeGW1}, the populations $t_1,\dots,t_{\f N}$ of these $\f N$ islands are chosen independently. Each has a probability $p_N \sim \l / N$ to be fertile,  and the population of any other island has law $\tpi_N$, and $N \tpi_N \Rightarrow \l \mu^c$. Hence, the measure
\[
\sum_{i = 1}^{\f N} \dl_{t_i}
\]
is approximately a Poisson measure with intensity
\[
\f \l \left ( \mu^c + \dl_c \right ).
\]
By the superposition property of Poisson measures, this can be reformulated by saying that the descendants of a fertile individual with fertility $\f$ are
\begin{itemize}
	\item either fertile, and there is approximately a Poissonian number with parameter $\f \l$ of those;
	\item or not fertile, and those contribute to the empirical measure as approximately a Poisson random measure with intensity $\f \l \mu^c$ .
\end{itemize}
Finally, since the initial number of individuals is $N$, we may link them to a virtual ancestor of fertility $1$.

This invites us to introduce the following measure $\eta$. We first construct a Galton-Watson tree $T$ with fertilities\footnote{This is, once again, just another way to speak of types, but we would rather avoid the confusion.} in $(0,\pinf)$. The number of children of an individual with fertility $\f$ is Poisson with parameter $\f \l$, and each child has a fertility chosen independently according to $\th$. We start from an individual with fertility $1$, and denote $\f_u$ the fertility of $u \in T$.

Obviously, if we forget about the types, $T$ has the same law as a Galton-Watson tree, constructed as follows:
\begin{itemize}
	\item the reproduction law of the ancestor is Poisson with parameter $\l$;
	\item the other individuals have reproduction law $\gTh$, which is a Cox law: it is a mixture of Poisson laws, where the random parameter is chosen as $\l \sqrt{c} \s W_1^+$, where we recall that $\sqrt{c} \s W_1^+$ has law $\th$.
\end{itemize}
The point of this alternative construction is to note that $\l \sqrt{c} \s W_1^+ = \sqrt{2/\pi} W_1^+$, so $\gTh \neq \dl_1$ and has mean $1$, so that $T$ is a.s. finite. For matters concerning the construction of these variables and the measurability of the functions and variables we consider, we refer to \cite[ch. III]{Harris} and \cite[ch. VI]{Atney}.

Finally, we define a measure $\eta$ as follows: consider $(\nu_u(\f), \, \f \geq 0)_{u \in \cU}$ an i.i.d. family, consisting of collections indexed by $\R^+$ of random variables, such that, for each $u \in \cU$ and $\f \geq 0$, $\nu_u(\f)$ is a Poisson measure with intensity $\f \l \mu^c$. This can for instance be constructed using a family of i.i.d. Poisson processes indexed by $\cU$, see Section \ref{sec:threelemmas}. Then we define
\[
\eta = \sum_{u \in T} ( \nu_u(\f_u) + \dl_c) - \dl_c = \sum_{u \in T} \nu_u(\f_u) + (\# T - 1) \dl_c.
\]
Subtracting $\dl_c$ is just to take into account that the ancestor is a virtual one. We may now state our result.

\begin{thm} \label{th:conv1}
Under the assumption \eqref{eq:hyprN1}, the sequence $(\PN_N(r_N))$ converges in distribution in $\cM^+$ to a random measure with the same law as $\eta$. 
\end{thm}

\subsection{Proof}

We first fix a sequence $(r_N)$ such that \eqref{eq:hyprN1} is verified and proceed as explained in Section \ref{sec:technique} by first proving the tightness.

\begin{lemma} \label{lem:tightness}
The sequence $(\PN_N(r_N))$ is tight in $\cM^+$.
\end{lemma}

\begin{proof}
This fact can be proven as in Lemma 5 in \cite{BertoinAR}. Indeed, one can readily check that, since our test functions have compact support in $(0,\pinf)$, it follows from the tightness of $\la \PN_N(r_N), \Id \ra$. But the latter is the total population of a Galton-Watson forest started from $N$ ancestors, renormalized by $1/N^2$, and it is well-known that this converges to the total population of a Feller diffusion, see e.g. \cite{LeGallBook}.
\end{proof}

Our next step is to  derive here an equation solved by the cumulant of $\PN_N(r_N)$, given by
\[
\k_N(f) = - \ln \E \left ( \exp - \la \PN_N(r_N) , f \ra \right ) = - N \ln \E \left ( \exp - \la \PN(r_N) , f \ra \right ),
\]
where we recall that $\PN_N(r_N)$ is the sum of $N$ independent copies of $\PN(r_N)$.

\begin{lemma} \label{lem:eqcumu}
The cumulant $\k_N(f)$ solves the following equation
\begin{equation}\label{eq:cumuN}
\exp - \k_N(f) = \E \left ( \exp - \left ( f \left ( \frac{P_N}{N^2} \right ) + \frac{C_N}{N} \k_N(f) \right ) \right )^N
\end{equation}
for every $f \in C_K$.
\end{lemma}

\begin{proof}
Lemma \ref{lem:treeGW1} and the branching property show that $(Z_{0,p}(r_N))_{p=1,\dots,r_N}$ has the same law as $(\unn{P_N=1},\dots,\unn{P_N=r_N})$, and that, conditionally on $P_N$ and $C_N$, $(Z_{i,\cdot}(r_N))_{i \geq 1}$ is independent from $Z_{0,\cdot}$ and has the same law as the sum of $C_N$ independent copies of $(Z_{i,\cdot}(r_N))_{i \geq 0}$. Hence, conditioning on $P_N$ and $C_N$, we may write
\begin{align*}
\exp - \k_N(f) & = \E \left ( \exp - \la \PN_N(r_N) , f \ra \right ) \\
& = \E \left ( \exp - \sum_{i \geq 0} \sum_{p=1}^{r_N} Z_{i,p}(r_N) f(p/N^2) \right )^N \\
& = \E \left ( \exp - \sum_{p=1}^{r_N} Z_{0,p}(r_N) f(p/N^2) \times \E \left ( \exp - \sum_{i \geq 0} \sum_{p=1}^{r_N} Z_{i,p}(r_N) f(p/N^2) \right )^{C_N} \right )^N \\
& = \E \left ( \exp - \sum_{p=1}^{r_N} \unn{P_N = p} f(p/N^2) \times \E ( \exp - \PN(r_N) )^{C_N} \right )^N \\
& = \E \left ( \exp - f(P_N/N^2) \times \E ( \exp - \PN_N(r_N) )^{C_N/N} \right )^N \\
& = \E \left ( \exp - f(P_N/N^2) \times (\exp - \k_N(f))^{C_N/N} \right )^N
\end{align*}
and the result follows.
\end{proof}

This lemma allows us to give an equation solved by any limit point of $(\k_N(f))$ (which exist by Lemma \ref{lem:tightness}).

\begin{prop} \label{prop:limcumu}
For $f \in C_K$, any limit point $\k(f)$ of $(\k_N(f))$ is the unique solution to the following equation
\begin{equation} \label{eq:cumu}
\exp - k(f)  = \l \left (\int_0^{\pinf} (1-e^{-f(x)}) \mu^c(\ddx) + \int_0^{\pinf}  \left ( 1 - e^{-f(c)- \k(f) x} \right ) \th(\ddx) \right ).
\end{equation}
\end{prop}
\noindent Note that the first term on the right-hand side is the cumulant of a Poisson random measure with intensity $\mu^c$.
\begin{proof}
We assume for simplicity that $(\k_N(f))$ converges to some $\k(f)$. Let us investigate the behavior of the right-hand term of Equation \eqref{eq:cumuN}. The expectation therein tends to $1$, so we just have to study
\[
\begin{split}
N \E \left (1 - e^{-f (P_N/N^2) - \k_N(f) C_N/N} \right ) = & N \E \left ( \left ( 1 - e^{-f(P_N/N^2)} \right ) \unn{P_N < r_N} \right ) \\
& + N \E \left ( \left ( 1 - e^{-f (r_N/N^2)-\k_N(f) C_N/N} \right ) \unn{P_N = r_N} \right ) .
\end{split}
\]

\begin{enumerate}[fullwidth]
\item We may rewrite the first term on the RHS as
\[
\P(P_N < r_N) N \E \left ( \left ( 1 - e^{-f(P_N/N^2)} \right ) \middle | \unn{P_N < r_N} \right ) = (1-p_N) \int_0^{\pinf} \left ( 1 - e^{-f(x)} \right ) N \tpiN(\ddx).
\]
Since $x \mapsto 1 - e^{-f(x)} \in C_K$, Lemma \ref{lem:convmeas} ensures that this tends to
\[
\int_0^{\pinf} \left ( 1 - e^{-f(x)} \right ) \mu^c(\ddx).
\]

\item Now, the second term on the RHS is
\begin{align*}
& N \E \left ( \left ( 1 - e^{-f (r_N/N^2) - \k_N(f) C_N/N} \right ) \unn{P_N = r_N} \right ) \\
& = N p_N \int_0^{\pinf} \left ( 1 - e^{-f (r_N/N^2) - \k_N(f) x} \right ) \tgN(\ddx) \\
& := N p_N \int_0^{\pinf} g_N(x) \tgN(\ddx).
\end{align*}
The sequence $(g_N)$ converges pointwise to
\[
g \, : \, \mapsto 1 - e^{-f(c)-\k(f) x}
\]
and all the $g_N$'s are bounded and uniformly Lipschitz-continuous. Since $\tgN \to \th$ by Lemma \ref{lem:convmeas}, it is then straightforward that
\[
\int_0^{\pinf} g_N(x) \tgN(\ddx) \to \int_0^{\pinf} g(x) \th(\ddx).
\]
The result then follows after from the convergence $N p_N \to \l$.

\item It only remains to prove that the equation has a unique solution. Taking $x = e^{-\k(f)}$ has the unknown, we may rewrite it
\[
x = A - B \E (x^{\s \sqrt{c} W_1^+}),
\]
with $B > 0$. The function $x \mapsto x - A + B \E (x^{\s \sqrt{c} W_1^+})$ has derivative
\[
1 + B \s \sqrt{c} \E ( W_1^+ x^{\s \sqrt{c} W_1^+ - 1}).
\]
This last quantity is positive, whence the result follows. \qedhere
\end{enumerate}
\end{proof}

The last step is to check that the measure $\eta$ defined in Section \ref{sec:heuristics} has the same law as the limit measure we just obtained, and to this end, one only need to prove the following.

\begin{prop} \label{prop:samecumu}
The cumulant of the random measure $\eta$ solves Equation \eqref{eq:cumu}.
\end{prop}

\begin{proof}
We denote $T(\f)$ a tree constructed as in Section \ref{sec:heuristics}, but starting instead from an ancestor of fertility $\f$, and let $\eta(\f)$ be the corresponding measure. For $g \in C_K$, consider
\[
\phi(\f) = \E ( \exp - \la \eta + \dl_c, g \ra ).
\]
Note that we count $\dl_c$ for the virtual ancestor to simplify the computations (in particular, the cumulant of $\eta \eqlaw \eta(1)$ is $- \ln \phi(1) - g(c)$). Indeed, the branching property shows that in the tree $T(\f)$, conditionally on the number $k$ of individuals at the first generation, which is Poissonian with parameter $\l \f$, each subtree rooted at the first generation accounts for a measure $\eta^j$ such that 
\begin{itemize}
	\item the $\eta^j$'s are independent from the measure $\nu_{\emp}(f_{\emp}) + \dl_c$ generated by the root,
	\item $\eta^j$ has the same law as $\eta(\f_j) + \dl_c$, where the $\f_j$'s are i.i.d. with law $\theta$.
\end{itemize}
Hence, conditioning on the number of children at the first generation and using the exponential formula for Poisson measures, we have
\begin{align*}
\phi(\f) & = \E (\exp - \la \eta_{\emp}(\f) + \dl_c , g \ra ) \sum_{k \geq 0} e^{-\l \f} \frac{(\l \f)^k}{k!} \left ( \int_0^{\pinf} \E ( \exp - \la \eta(s) + \dl_c,g \ra ) \th(\dds) \right)^k \\
 & = e^{- \l \f} e^{-g(c)} \E ( \exp - \la \eta_{\emp}(\f), g \ra ) \sum_{k \geq 0} \frac{(\l \f)^k}{k!} \left ( \int_0^{\pinf} \phi(s) \th(\dds) \right)^k  \\
 & = e^{- \l \f} e^{-g(c)} \exp - \f \l \int_0^{\pinf} (1 - e^{-g(s)}) \mu^c(\dds) \exp \left ( \l \f \int_0^{\pinf} \phi(s) \th(\dds) \right )  \\
 & = \exp - \left ( g(c) + \f \l \int_0^{\pinf} (1 - e^{-g(s)}) \mu^c(\dds) + \l \f \left ( 1 - \int_0^{\pinf} \phi(s) \th(\dds) \right ) \right ) .
\end{align*}
This shows that
\[
\phi(\f)e^{g(c)} = \left ( \phi(1) e^{g(c)} \right )^{\f}.
\]
Plugging this in the above formula readily shows that $\Phi := - \ln \phi(1)$ solves 
\[
\Phi = g(c) + \l \int_0^{\pinf}c (1-e^{-g(s)}) \mu(\dds) + \l \left ( 1 - \int_0^{\pinf} e^{-g(c)-s (\Phi-g(c))} \th(\dds) \right )
\]
so that $\Phi - g(c)$, which is the cumulant of $\eta$, also solves \eqref{eq:cumu}.
\end{proof}

\subsection{Genealogy of the islands}

\subsubsection{Introduction}

The arguments and heuristics mentioned in Section \ref{sec:heuristics} should make it clear that a result concerning the genealogy of the island could be obtained, which is lost when considering merely the empirical measure $\PN_N(r_N)$. We shall thus now give an idea of what the ``genealogical tree'' of the islands looks like at the limit when $N \to \pinf$. Obviously, this tree is infinite, since e.g. we start from $N \to \pinf$ islands. On the other hand, the heuristics of Section \ref{sec:heuristics} suggest that the tree consisting only of fertile islands should be finite, more precisely be a Galton-Watson tree with the (critical) Cox reproduction law $\gTh$, so that the whole genealogical tree of the islands has a.s. finite height (but infinite width). Inspired by Definition 1 in \cite{BertoinToA}, we will now introduce the definition of a tree-indexed Continuous State Branching Process with types\footnote{The types being what we called populations, but we wish to give the most general definition here.}.

Let us define in the following $\cMinf$ the subset of $\cM^+$ of measures which integrate $1$ at infinity. In particular, if $\mu \in \cMinf$, we can rank the atoms of a Poisson measure with intensity $\mu$ in the decreasing order, which allows the following definition to make sense. If $\mu$ is finite, there is only a finite number of such atoms, and we shall always complete this decreasing sequence with an infinite sequence of zeros.

\begin{defn}
Consider a measurable space $T$, and a family of $\s$-finite measures $(\rho_t)_{t \in T}$ on $T \times (0,\pinf)$, such that, for every $t \in T$, $\rho_t ( T \times \cdot) \in \cMinf$. Fix $t_0 \in T$ and $\f_0 \geq 0$. A tree-indexed CSBP with types, with reproduction laws $(\rho_t)_{t \in T}$, started from $(t_0,\f_0)$, is a process $(\cZ_u)_{u \in \cU}$ indexed by the universal tree $\cU$, with values in $T \times (0,\pinf)$, such that
\begin{itemize}
	\item $\cZ_{\emp}= (t_0,\f_0)$ a.s.;
	\item for every $k \in \Z^+$, conditionally on $(\cZ_u, u \in \cU, |u| \leq k)$,
		\begin{itemize}
			\item the sequences $(\cZ_{u \, j})_{j \in \N}$, for $|u| = k$, are independent;
			\item for $|u| = k$, writing $\cZ_u = (t,\f)$, the sequence $(\cZ_{u \, j})_{j \in \N}$ is distributed as the family of the atoms of a Poisson measure with intensity $\f \rho_t$, where atoms are repeated according to their multiplicity and ranked in the decreasing order of their second coordinate.
		\end{itemize}
\end{itemize}

\end{defn}

Notice that the branching property holds with respect to the second variable, that is, the independent sum of a CSBP with parameters $(t_0,\f_0,(\rho_t))$ and one with parameters $(t_0,\f'_0,(\rho_t))$ is a CSBP with parameters $(t_0,\f_0+\f'_0,(\rho_t))$.

\begin{figure}[htb]
\centering
\includegraphics[width= \columnwidth]{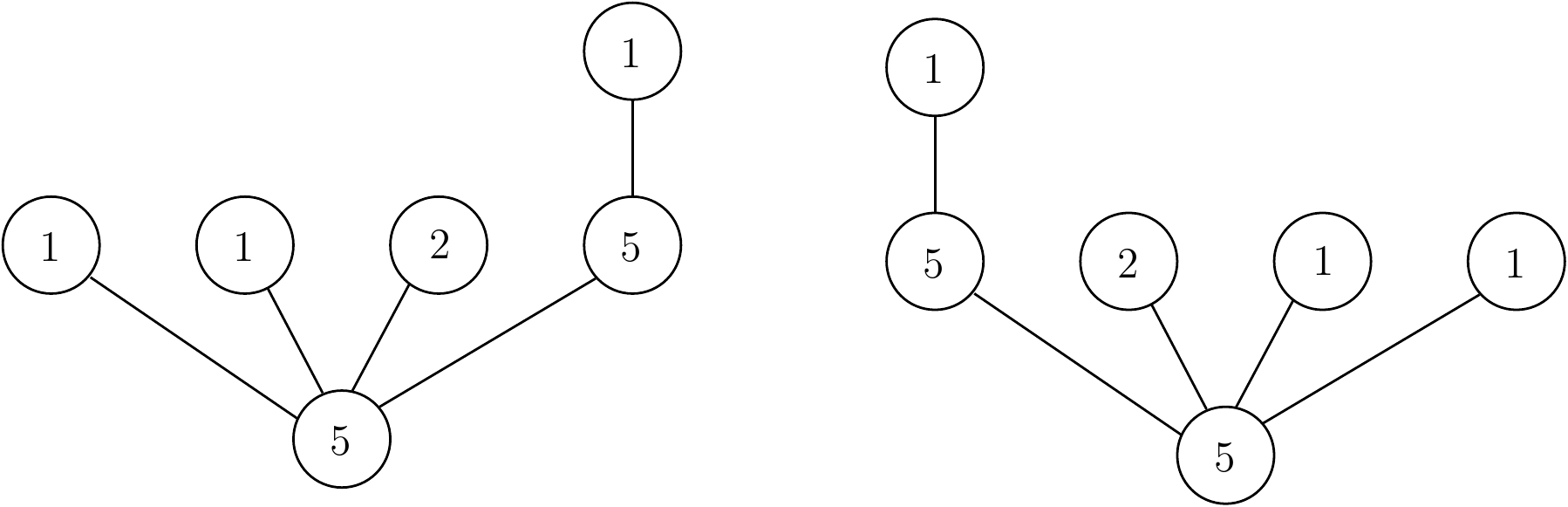}
\caption{The tree of isles from Figure \ref{fig:tree_of_isles_2}, and its reordering.}
\label{fig:tree_of_isles_reordered}
\end{figure}

Let us define the object we shall study. As before, we now call the types ``population'', the vertices ``islands'' and the descendants ``colonies''. We have a forest of i.i.d. trees of isles 
\[
\left ( \A^1(r_N),\dots,\A^N(r_N) \right ).
\]
We root these trees at $1, \dots, N$, and link them to $\emp$, to which we give population $r_N$. We call $\FN$ the tree obtained and once again, we give population 0 to the islands in $\cU \bsl \FN$. We define the tree $\FNr$ by reordering the colonies of each island, along with their subtree, in the decreasing order of their population, leaving unchanged the initial order if ties occur: see Figure \ref{fig:tree_of_isles_reordered}. The following result deals with its convergence. We may actually prove a slightly stronger result and to this end, let us first introduce the $\cU$-indexed processes $\cZN$ and $\cZNr$ defined by
\[
\cZN_u = \left ( \FN_u / N^2, k_u(\FN)/N \right ), \quad \cZNr_u = \left ( \FNr_u / N^2, k_u(\FNr)/N \right ), \quad u \in \cU,
\]
where we recall that $k_u(\FN)$ is the number of colonies of $u$ in $\FN$, or in formulas,
\[
k_u(\FN) = \# \{ j \in \N, \FN_{u \, j} \neq 0 \}.
\]
We shall prove the following result, where we denote $q_1$ the projection on the first coordinate.

\begin{thm} \label{th:convtree}
The process $(\cZNr)$ converges as $n \to \pinf$, in the sense of finite dimensional distributions, to $\cZ$, where $\cZ$ is a tree-indexed CSBP with types, started from $(c,1)$, with reproduction law $\l (\mu^c + \dl_c) \otimes \th$ for the type $c$, and $\dl_0 \otimes \dl_0$ for the others. In particular $(\FNr)$ converges to $q_1(\cZ)$.
\end{thm}

It is worth noticing two features of this result. First, one cannot construct the limit of $(\FNr)$ directly. To this end, we need first to construct $\cZ$, and then project it on the first coordinate; but the knowledge of the second coordinate is necessary to get the whole process. This is why we prove the more general convergence of the two-coordinate process in order to obtain the convergence of $(\FNr)$.

Note also that this result agrees with Theorem \ref{th:conv1}, in that the measure generated by the tree $\cZ$ (except for his root)
\[
\eta' := \sum_{u \in \cU^*} \dl_{q_1(\cZ_u)} \unn{q_1(\cZ_u) \neq 0}
\]
has the same law as the measures of Theorem \ref{th:conv1}. Indeed, similar calculations as in the proof below and the proof of Proposition \ref{prop:samecumu} can be carried out to show that the cumulant of $\eta'$ solves Equation \eqref{eq:cumu}. However, since we only show a result dealing with the convergence of finite-dimensional marginals, we cannot deduce Theorem \ref{th:conv1} from Theorem \ref{th:convtree}. Doing this would require to introduce a relevant topology on the tree-indexed CSBP (with types or not) and prove the tightness results associated. This technical and long detour would not bring, we believe, much more understanding of the model.

\subsubsection{Some results about Poisson random measures} \label{sec:threelemmas}

Let us start with three preliminary lemmas. The first is essentially a classical fact (see e.g. \cite{BertoinToA}), and merely rephrases Theorem 16.18 in \cite{Kallenberg}.

\begin{lemma} \label{lem:convpoisson}
Let $(\nuN)$ be a sequence of probability measures $(0,\pinf)$, and assume that
\[
N \nuN \Rightarrow \nu
\]
as $N \to \pinf$, for some $\nu \in \cMinf$. Let, for each $N$, $(Y^{(N)}_i)_{i \geq 0}$ a sequence of i.i.d. random variables with law $\nuN$, and take, for $\f > 0$, $(\rmaN_1(\f),\dots,\rmaN_{\f N}(\f))$ the reordering of $(Y^{(N)}_i)_{i = 1 \dots \f N}$ in the decreasing order. Then, for every fixed $k \geq 1$,
\[
(\rmaN_1(\f),\dots,\rmaN_k(\f)) \to (\rma_1(\f),\dots,\rma_k(\f)),
\]
where $(\rma_i(\f))_{i \geq 0}$ is the reordering in the decreasing order of the atoms of a Poisson random measure with intensity $\f \nu$.
\end{lemma}

Consider in the following $\nu \in \cMinf$, and let $(\rma_i(\f))_{i \geq 1}$ as in the previous statement. To make sense of the coming results, and since it is of use in the proofs, let us recall how we can construct a Poisson measure with intensity $\f \nu$ in a measurable way. Let $(A_i)_{i \geq 1}$ be a partition of $\R^+$ in measurable subsets of finite $\nu$-measure. Define for $i \geq 1$ such that $\nu(A_i) > 0$,
\[
\l_i = \nu(A_i), \quad \nu_i = \nu(\cdot \cap A_i)/\l_i.
\]
Independently for each $i$, let $(N^i(\f))_{\f \geq 0}$ be a Poisson process with intensity $\l_i$ and $(X^i_j)_{j \geq 1}$ a sequence of i.i.d. random variables with law $\nu_i$, and define
\[
\xi^i(\f) = \sum_{j=1}^{N^i(\f)} \dl_{X^i_j}, \quad \xi = \sum_{i \geq 1} \xi^i.
\]
Then $\xi(\f)$ is a Poisson random measure with intensity $\f \nu$.

\begin{lemma} \label{lem:contpoisson}
For every continuous $g \colon \R^+ \to \R$ with compact support and every $r \in \N$, the mapping
\[
\f \mapsto \E ( g ( \rma_r(\f)))
\]
is continuous on $\R^+$.
\end{lemma}

\begin{proof}
Let $\eps > 0$. The measure $\nu$ is in $\cMinf$, so $\a := \nu( ( \eps, \pinf ) )$ is finite. Then $\rma_r(\f) > \eps$ if and only if there are $r$ or more atoms of the Poisson measure in $(\eps,\pinf)$, so
\[
\P ( \rma_r (\f) \in (\eps , \pinf) ) = \sum_{k \geq r} e^{- \f \a} \frac{(\f \a)^k}{k!}
\]
which is clearly continuous in $\f$. Then, the complimentary probability $P(\rma_r(\f) \in [0, \eps))$ is also continuous. The result thus holds for any indicator function, and the result follows by standard approximations.
\end{proof}

\begin{lemma} \label{lem:unifconvpoisson}
In the notation of Lemma \ref{lem:convpoisson}, for every continuous $g \colon \R^+ \to \R$ with bounded variation and compact support, and every $r \in \N$,
\[
\E(g(\rmaN_r(\f))) \to \E(g(\rma_r(\f)))
\]
uniformly on the compact sets of $\R^+$.
\end{lemma}
\noindent This is in particular true, and that is all we will use, for a Lipschitz-continuous $g$ with compact support.
\begin{proof}
The function $g$ has bounded variation, so it may be written as $g = g^+ - g^-$, where $g^+$ and $g^-$ are nondecreasing and continuous. Take $A > 0$ and consider the mappings
\[
\phi_n^{\pm} : \f \mapsto \E(g^{\pm}(\rmaN_r(\f)))
\]
on $[0,A]$. It is obvious, in the construction above, that $\f \mapsto \rmaN_r(\f)$ is nondecreasing, thus so do $\phi_n^{\pm}$. Lemma \ref{lem:convpoisson} ensures that $\rmaN_r(\f) \to \rma_r(\f)$, so by dominated convergence, $\phi_n^{\pm}$ converges simply to $\phi^{\pm}$, where
\[
\phi^{\pm}(\f) = \E(g^{\pm}(\rma_r(\f))).
\]
Now, Lemma \ref{lem:contpoisson} shows that $\phi^{\pm}$ are continuous, so the result follows from Dini's second theorem.
\end{proof}

\subsubsection{Proof of Theorem \ref{th:convtree}}

The proof will rely heavily on Lemmas \ref{lem:treeGW1} and \ref{lem:convmeas}, which we will use without further notice. 
\begin{enumerate}[fullwidth]
	\item Let us start with some preliminary definitions. Recall that we say that $u \in \FN$ if $\FN_u \neq 0$. For $u \in \FN$, $\FN_u$ is called its population, instead of type, and $k_u(\FN)/N$ its fertility. An island in $\FN$ of population $\rN := r_N/N^2$ is said to be fertile. For a function $\cZ$ indexed by $\cU$ and $u \in \cU$, $\cZ^{u +}$ is the function obtained by shifting the subtree rooted at $u$ back to $\emp$, so for instance, $\cZ^{u+}_{\emp} = \cZ_u$.

Let $\cF$ be the set of functions from $\cU$ to $(\R^+)^2$. For $k \in \Z^+$, we define $E_k$ the set of functions $g$ from $\cF$ to $\R^+$ which can be written as
\begin{equation} \label{Ek}
g(\cT) = g_{\emp} (\cT_{\emp}) g_1(\cT_1) \dots g_s(\cT_s) g_{1 \, 1}(\cT_{1 \, 1}) \dots g_{s \dots s}(\cT_{s \dots s})
\end{equation}
for every $\cT \in \cF$, for some $s \in \N$, where the $g_u$'s are Lipschitz-continuous with compact support in $(\R^+)^2$, and the last index consists of $k$ letters $s$. In particular, such a $g$ depends only on the $k$ first generations. If $\cT$ consists only on a root with population $p$, that is if $\cT_{\emp} = (p,0)$ and $\cT_u = (0,0)$ for $u \in \cU^*$, then we denote $g(p) := g(\cT) = g_{\emp}(p,0)$.

  \item In the sequel, we will need to consider $\FN(\f)$ a tree constructed as $\FN$, but with $\f N$ islands at generation $1$ (so $\FN$ has the law of $\FN(1)$). We say that this tree has fertility $\f$. More generally, we may take $\f$ to be random with some law $\nu$, and we write $\FN(\nu)$ for the corresponding tree. We may also construct as above $\FNr(\nu)$, $\cZN(\nu)$ and $\cZNr(\nu)$. Rather than $\E(g(\cZNr(\nu)))$, we may write $E_{\nu}(g(\cZN))$. We need these trees because of the following observation: consider a fertile island $u \in \FN(\nu)$, $u \neq \emp$. Then $u$ has a fertility chosen according to $\tgN$, so the branching property shows that $\F^{(N),u+}$ has the law of $\FN(\tgN)$, and $\cZ^{(N),u+}$ has the law of $\cZN(\tgN)$.

Now, take an island $u \in \FN(\nu)$, and assume it is fertile with fertility $\f$. Amongst $u \, 1, \dots, u \: \f n$, each has independently a probability $p_N$ to be fertile, so the number of fertile islands is binomial with parameters $\f N$ and $p_N$. The fertile trees have the law of $\FN(\gN)$, and the other ones have the law of the trivial tree with only a root and population chosen according to $\tpiN$. Let us condition on $u$ having $k$ fertile colonies, and consider the populations $\YN_1(\f),\dots,\YN_{\f N - k}(\f)$ of the non-fertile children. $(\YN_1(\f),\dots,\YN_{\f N - k}(\f))$ is then an i.i.d. sequence with law $\tpiN$.

Let us reorder the tree and take a look at the first $s$ colonies of the root of $\FNr(\nu)$. Take $\XNf$ a binomial variable with parameters $\f N$ and $p_N$. For $1 \leq i \leq s$, the analysis above shows that with probability $\P(\XNf \wedge s = i)$, the following happens:
\begin{itemize}
	\item the $i$ first islands are fertile;
	\item the population of the $s-i$ following islands has the law of $(\rmaN_{1}(\f),\dots,\rmaN_{s-i}(\f))$ where we denote $(\rmaN_1(\f),\dots,\rmaN_{\f N - i}(\f))$ the reordering in decreasing order of a $(\f N - i)$-sample of law $\tpiN$.
\end{itemize}

\item Let us now reason by induction, the assumption $\cP_r$ at rank $r$ being that for every $f \in E_k$ and every laws $\nuN, \nu$ on $\R^+$ such that $\nuN \convlaw \nu$, we have
\begin{equation} \label{rec}
\E_{\nuN} \left ( f \left ( \cZNr \right ) \right ) \to \E_{\nu} \left ( f (\cZ) \right ),
\end{equation}
where $\cZ$ is defined in Theorem \ref{th:convtree}. This obviously implies the result by picking $\nuN = \nu = \dl_1$.

$\cP_0$ is easy to check by weak convergence of $\nuN$ to $\nu$ and uniform continuity of $\f_{\emp}$. So assume $\cP_{r-1}$, and take $g \in E_r$, which we write
\[
g(\cZ) = g(\cZ_{\emp}) g_1(\cZ^{1+}) \dots g_s(\cZ^{s+}),
\]
with $g_1,\dots,g_s \in E_{r-1}$.

By the analysis above, we may condition the tree $\FN(\nuN)$ on the fertility of the root and its number of fertile children, to obtain
\begin{align*}
\E_{\nuN}(g(\cZN)) & = \int \nuN(\df) g_{\emp}(\rN,\f) \times \\
& \left ( \sum_{i=0}^{s-1} \P(\XNf = i) \E_{\tgN}(g_1(\cZN)) \dots \E_{\tgN}(g_i(\cZN)) \right. \times \\
& \quad \E(g_{i+1}(\rmaN_1(\f))) \dots \E(g_s(\rmaN_{s-i}(\f))) \\
& + \left. \sum_{i \geq s} \P(\XNf = i) \E_{\tgN}(g_1(\cZN)) \dots \E_{\tgN}(g_s(\cZN)) \right ) \\
& = \sum_{i=0}^{s-1} \E_{\tgN}(g_1(\cZN)) \dots \E_{\tgN}(g_i(\cZN)) \times \\
& \int \nuN(\df) g_{\emp}(\rN,\f)\P(\XNf = i) \E(g_{i+1}(\rmaN_1(\f))) \dots \E(g_s(\rmaN_{s-i}(\f))) \\
& + \E_{\tgN}(g_1(\cZN)) \dots \E_{\tgN}(g_p(\cZN)) \int g_{\emp}(\rN,\f) \P(\XNf \geq i) \nuN(\df).
\end{align*}
By the induction hypothesis $\cP_{r-1}$ and Lemma \ref{lem:convmeas},
\[
\E_{\tgN}(g_i(\cZN)) \to \E_{\th}(g_i(\cZ)).
\]
Define $\Xf$ a variable with Poisson law with parameter $\l \f$. Let us prove that the quantity
\begin{equation} \label{eq:lastterm}
\int \nuN(\df) g_{\emp}(\rN,\f)\P(\XNf = i) \E(g_{i+1}(\rmaN_1(\f))) \dots \E(g_s(\rmaN_{s-i}(\f)))
\end{equation}
converges to
\begin{equation} \label{eq:limlastterm}
\int \nu(\df) g_{\emp}(c,\f)\P(\Xf = i) \E(g_{i+1}(\rma_1(\f))) \dots \E(g_s(\rma_{s-i}(\f)))
\end{equation}
where $(\rma_i(\f))_{i \geq 1}$ is the reordering in the decreasing order of the atoms of a Poisson measure with intensity $\f \mu^c$. According to Lemmas \ref{lem:convmeas} and \ref{lem:unifconvpoisson}, we obtain that
\[
\E(g_{i+j}(\rmaN_j(\f))) \to \E(g_{i+j}(\rma_j(\f))), \quad j=1 \dots s-i,
\]
uniformly on the compact sets of $\R^+$. Using this fact and the weak convergence of $\nuN$ to $\nu$, it is easy to see that \eqref{eq:lastterm} has the same limit, if any, as
\begin{equation} \label{eq:lastterm2}
\int \nuN(\df) g_{\emp}(\rN,\f)\P(\XNf = i) \E(g_{i+1}(\rma_1(\f))) \dots \E(g_s(\rma_{s-i}(\f))).
\end{equation}
The difference between \eqref{eq:lastterm2} and \eqref{eq:limlastterm} is bounded, up to a constant, by
\[
\begin{split}
& \int \left |g_{\emp} (\rN,\f) - g_{\emp}(c,\f)\right | \nuN(\df) + \sup \left | g_{\emp} \right | \int_{\mathrm{supp} \: g_{\emp}} \nuN(\df) \left | \P(\XNf = i) - \P(\Xf = i) \right | \\
+ & \left | \int g_{\emp} (c,\f) \nuN(\df) - \int g_{\emp}(c,\f) \nu(\df) \right |.
\end{split}
\]
The first and last term tend to 0 by uniform continuity of $g_{\emp}$ and weak convergence of $\nuN$ to $\nu$. For the second one, Le Cam's inequality \cite{LeCam} gives
\[
\left | \P(\XNf = i) - \P(\Xf = i) \right | \leq \sum_{k=1}^{\f N} p_N^2 = \f O \left ( \frac1N \right ),
\]
where $O$ is uniform in $\f$, what ensures the expected convergence. By the same reasoning, we also have that
\[
\int g_{\emp}(\rN,\f) \P(\XNf \geq i) \nuN(\df) \to \int g_{\emp}(c,\f) \P(\Xf \geq i) \nu(\df).
\]
We hence get that $\E_{\nuN}(g(\cZN))$ tends to
\[
\begin{split}
& \sum_{i=0}^{s-1} \E_{\th}(g_1(\cZ)) \dots \E_{\th}(g_i(\cZ)) \times \int g_{\emp}(c,\f) \nu(\df) \E(g_{i+1}(\rma_1(\f))) \dots \E(g_s(\rma_{s-i}(\f))) \P(\Xf = i) \\
& + \E_{\th}(g_1(\cZ)) \dots \E_{\th}(g_s(\cZ)) \int g_{\emp}(c,\f) \P(\Xf \geq i) \nu(\df).
\end{split}
\]
By a similar computation as just done, we can see that this is precisely $\E_{\nu}(g(\cZ))$, what shows the result.
\end{enumerate}

\section{Second model: regrowing resources} \label{sec:secondmodel}

We shall now present our second model of migration under constraints, which is arguably more natural than the first one. The only difference is the migration rule: now, we shall assume that individuals will migrate when there are too many of them living at the same time on the same island. In other words, we assume that each island has regrowing resources, enough to feed $r$ people at the same time. If a birth happens when $r$ people coexist on a given island, the newborns will migrate each to a different virgin island, containing $r$ resources.

To make this model more natural, we will now consider continuous time trees, so that we can precisely tell which birth event forces certain individuals to migrate. Otherwise, imagine that at some generation, more than $r$ individuals coexist. We would then have to choose which children of which individual migrate. Doing this in a relevant way would force us to choose the migrant individuals uniformly at random, what would introduce some randomness in the construction of the tree of isles from a given (deterministic) tree, what we want to avoid.

\subsection{Tree of isles} \label{sec:treeofisles2}

Start from a finite continuous tree $\cT$, such that two events (birth or death) do not occur at the same time (which is a.s. the case for a continuous Galton-Watson tree). Once again, to define the tree of isles, we just need to define its migrant children. To this end, consider the process $(N_t)$ counting the number of people alive at time $t$, which we assume to be c\`adl\`ag. Define $\tau$ be the first time (if any) such that
\begin{itemize}
	\item for $t \in [0,\tau)$, $N_t \leq r$;
	\item $N_{\tau} > r$.
\end{itemize}
Then the $N_{\tau}- r$ rightmost individuals born at $\tau$ will be migrant children. After this, cut off the trees rooted at these individuals, and proceed identically with the pruned tree, until it is impossible to find other migrant children. See Figure \ref{fig:cont_tree} for an example.

\begin{figure}[htb]
\centering
\includegraphics[width=0.75 \columnwidth]{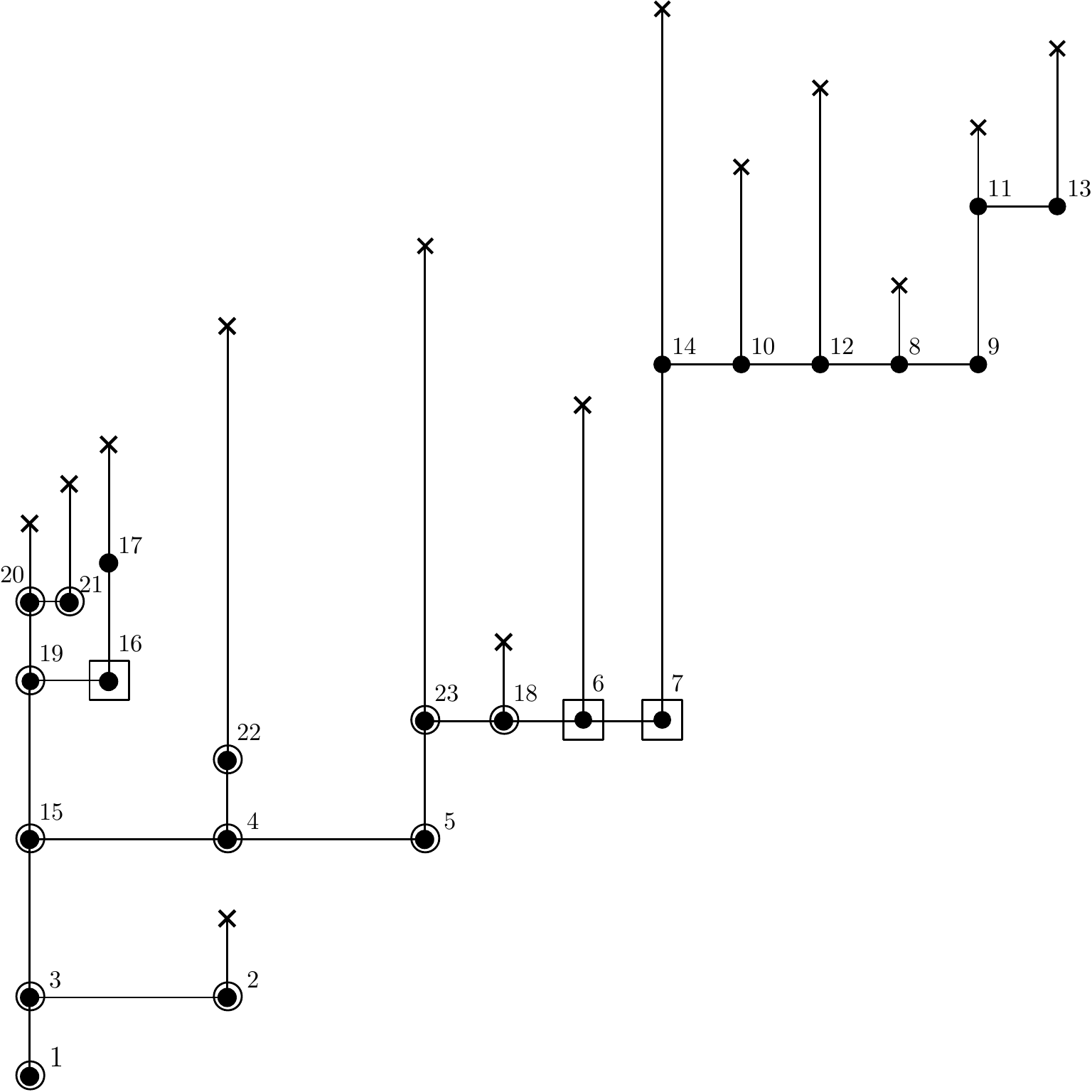}
\caption{A tree $\cT$ and its labeling, for $r=4$. The individuals in a circle remain on the island, those in a square migrate. Here $P_r(\cT) = 12$ and $C_r(\cT) = 3$.}
\label{fig:cont_tree}
\end{figure}

This thus allows us to construct the tree of isles as explained in Section \ref{sec:treeofisles}. An example is given in Figure \ref{fig:tree_of_isles_2}.

\begin{figure}[htb]
\centering
\includegraphics[width=0.25 \columnwidth]{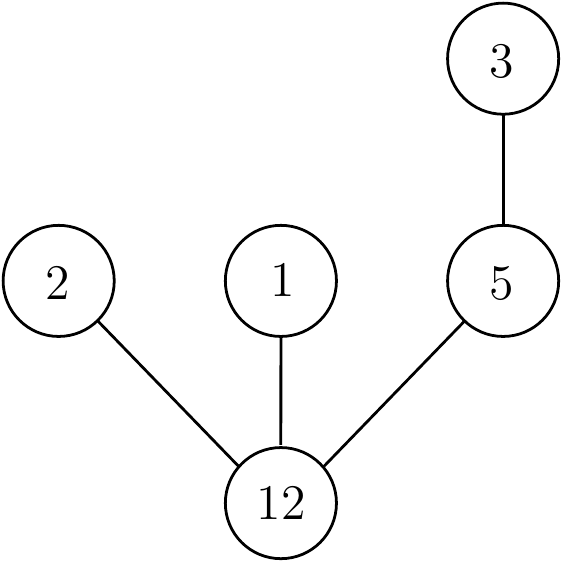}
\caption{The tree of isles $\cA(r)$ drawn from the tree $\cT$ of Figure \ref{fig:cont_tree}, for $r = 4$.}
\label{fig:tree_of_isles_2}
\end{figure}

We shall now explain how to label our tree so that the relevant information, namely its population $P(\cT)$ and number of colonies $C(\cT)$, can easily be read on the corresponding exploration process $(S_n)$. According to Section \ref{sec:explproc}, to define this labeling, we only need to precise the (Markovian) rule on how to choose the next vertex to be labeled.

We shall take the following rule $\cR$, defined for every line $\ell$. For such a line, consider the pruned tree $\cT_{\ell}$, for which we can define migrant individuals thanks to the above algorithm. Take $\ell = \ell^+ \cup \ell^0 \cup \ell^-$ a partition of $\ell$ in
\begin{itemize}
	\item the children of migrants individuals,
	\item the migrant individuals
	\item and the other individuals.
\end{itemize}
Then $\cR$ is defined as follows:
\begin{itemize}
	\item if $\ell^+ \neq \emp$, then $\cR(\ell)$ is the individual in $\ell^+$ which will die the sooner;
	\item if $\ell^+ = \emp$ and $\ell^0 \neq \emp$, then $\cR(\ell)$ is the individual in $\ell^0$ which will die the sooner;
	\item else $\cR(\ell)$ is the individual of $\ell^-$ which will die the sooner.
\end{itemize}
By construction, this is clearly a Markovian rule.

This algorithm is just a modification of the death-first search algorithm. Informally, we apply the latter until we observe more than $r$ people coexisting, say $r'$ of them. When this is the case, $r'-r$ newborns migrate, and for the sake of definiteness, we choose the $r' - r$ rightmost ones. We then explore their subtrees one after another by death-first search, starting with the migrant individual who dies first. When all of these subtrees are explored, we resume the exploration of the initial tree by death-first search, until we see other migrant individuals or the exploration is over. All of this is probably clearer on a picture, see Figure \ref{fig:cont_tree}.

To this labeling, we may thus associate a walk $(S_n)$. As shown on Figure \ref{fig:cont_expl_proc}, let us define
\[
\s_1(r) = 0, \quad \vsinf = \inf \{ n \geq 1, S_n = - 1 \}
\]
where as usual $\inf \emp = \pinf$, and successively 
\[
\vs_i(r) = \inf \{ n > \s_i, S_n > r - 1 \} \wedge \vsinf, \quad \s_{i+1}(r) = \inf \{ n > \vs_i, S_n = r - 1 \} \wedge \vsinf,
\]
so that the walk first hits $[r,\pinf)$ at $\vs_1(r)$, then makes excursions\footnote{We shall use this word in a very loose sense, only the precise definitions should be taken as rigorous. However, and thankfully, ``Brownian excursion'' will have its usual meaning.} above $r$ on each interval $[\vs_i(r),\s_{i+1}(r))$, and below $r$ on each interval $[\s_{i+1}(r),\vs_{i+1}(r))$, before hitting $-1$ at $\vsinf$. To have convenient formulas, we have let all the $\s_i(r)$ and $\vs_i(r)$ to be equal to $\vsinf$ after the walk has hit $-1$. Defines as well
\[
\cO_i(r) = (S_{\vs_i(r)} - (r-1)) \vee 0, \quad i \geq 1
\]
be the overshoot above level $r-1$ (and by definition $0$ after $-1$ is hit), and
\[
\ell_i(r) = \vs_i(r) - \s_i(r), \quad i \geq 1,
\]
so the sum of the $\ell_i(r)$'s is the time spent below $r - 1$ before hitting $-1$. The relation between the population and number of colonies of the initial island and the corresponding exploration process is the following.

\begin{figure}[htb]
\centering
\includegraphics[width=\columnwidth]{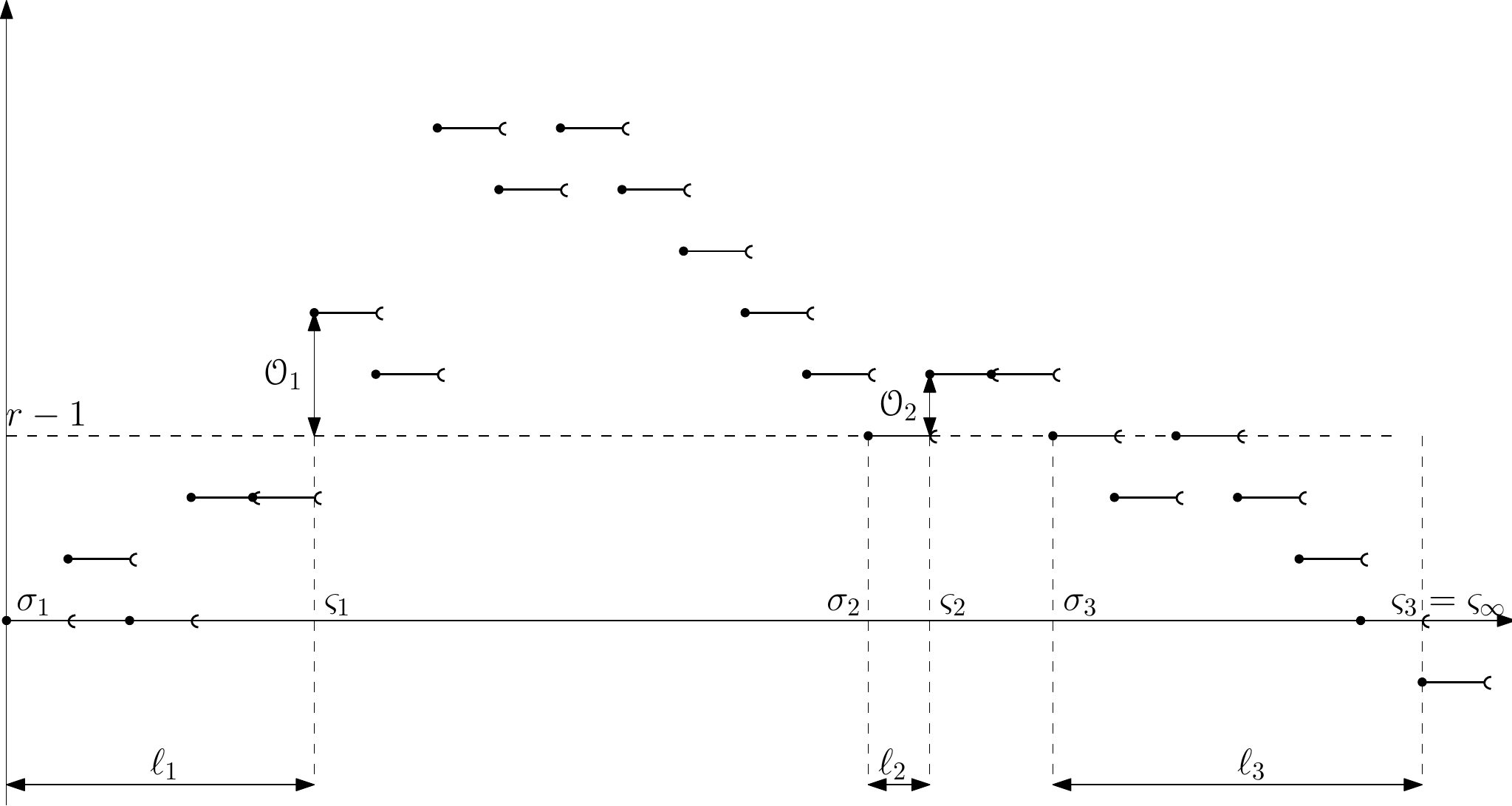}
\caption{The exploration process attached to the labeling of the tree of Figure \ref{fig:cont_tree}, for $r = 4$ (forgetting about the extra notation $r$). Check that $\cO_1(r)+ \cO_2(r) = 3 = C_r(\cT)$ and $\ell_1(r) + \ell_2(r) + \ell_3(r) = 12 = P_r(\cT)$.}
\label{fig:cont_expl_proc}
\end{figure}

\begin{lemma} \label{lem:popcolRW2}
The equalities
\[
P_r(\cT) = \sum_{i=1}^{\pinf} \ell_i(r), \quad C_r(\cT) = \sum_{i=1}^{\pinf} \cO_i(r)
\]
hold.
\end{lemma}
\noindent In words, $P_r(\cT)$ is the time spent by $(S_n)$ under $r-1$, and $C_r(\cT)$ the sum of the overshoots above level $r-1$, both before hitting $-1$. The quantities involved are easy to follow on Figures \ref{fig:cont_tree} and \ref{fig:cont_expl_proc}, and carefully following in parallel the walk and the tree on these examples should make the following proof quite obvious.
\begin{proof}
As explained above, our labeling algorithm works as follows. As long as there are no migrant children, it is just the death-first search algorithm. As we mentioned in Section \ref{sec:deathfirst}, $1 + S_i$ is precisely the number of individuals alive between the $(i-1)$-th and the $i$-th event. Hence, the first time $i$ that $(1 + S_n)$ is greater than $r$ is the first time when more than $r$ people coexist. The supernumerary $1 + S_i - r$ migrate, and this quantity is precisely the first overshoot $\cO_1(r)$, whereas the $i = \ell_1(r)$ first individuals visited remain on the initial island.

Then, we modify our algorithm to explore (by death-first search) the $\cO_1(r)$ subtrees of the migrant children. According to the first part of Lemma \ref{lem:explproc}, the first exploration goes from $1 + S_i$ to $1 + S_i - 1$ while remaining above $1 + S_i - 1$, the second from $1 + S_i - 1$ to $1 + S_i - 2$ and remains above $1 + S_i - 2$, \dots, and the $\cO_1(r)$-th and last from $1 + S_i - \cO_i(r) + 1 = r$ to $1 + S_i - \cO_i(r) = r - 1$ while remaining above $r-1$. Hence, when all these explorations are done, no new overshoot has been observed, and no time below $r$ has been cumulated. These subtrees have thus all been visited, and we may just as well cut them off.

The exploration of the initial tree then resumes, and the same argument can then be applied to the pruned tree, so we can thus conclude by a simple induction on the size of the tree.
\end{proof}

From now on, we replace $\cT$ by $\T$, a Galton-Watson tree with reproduction law $\rho$ and construct its tree of isles $\A(r)$ for $r$ resources. We also replace $(S_n)$ by an actual random walk defined on the whole of $\Z^+$, with step distribution $\trho$, so Lemma \ref{lem:popcolRW2} can be rewritten as an equality in law. Unlike the first model, an island with colonies can have an arbitrarily large population. Letting $\pi_r$ to be the law of the pair $(P_r(\T),C_r(\T))$, we cannot anymore write $\pi_r$ as (roughly) a product measure. It should be clear than an analogue to Lemmas \ref{lem:treeGW1} and \ref{lem:procGW1} holds. We call a $\N$-type Galton-Watson tree (resp. process) a multitype Galton-Watson tree (resp. process) with types in $\N$, and we still adopt the notation of Section \ref{sec:empmeas}.

\begin{lemma} \label{lem:GW2}
The tree of isles $\A(r)$ is a $\N$-type Galton-Watson tree, such that the population and number of colonies of each island has law $\pi_r$. Similarly, the process $(Z_{i,p}(r),\, p \geq 1)_{i \geq 0}$ is a $\N$-type Galton-Watson process, such that the type and number of children of each individual has law $\pi_r$. 
\end{lemma}

\subsection{Result}

As in Section \ref{sec:empmeas}, we now start from $N$ independent islands and $r_N$ resources, and employ the same notation. Let us first define, for a real function $X$ and a real $x$, $\tau_x(X)$ to be the hitting time of $(x,\pinf)$ by $X$, i.e.
\[
\tau_x(X) = \inf \{ t \geq 0, X_t > x \}, \quad \inf \emp = \pinf.
\]
Let $C_N = C_{r_N}(\T)$ and $P_N = P_{r_N}(\T)$. From Lemma \ref{lem:popcolRW2}, the probability that an island has colonies is
\begin{equation} \label{eq:probacol}
\P(C_N > 0) = \P(\tau_{r_N-1}(S) < \vsinf) \sim  \frac{1}{r_N}.
\end{equation}
The last equivalent is an extension of the classical gambler's ruin estimates, and follows e.g. from Theorem 2, p. 18, in \cite{Takacs} and a Tauberian theorem (p.203, 204 of \cite{Takacs}). We shall thus assume that, for some $c > 0$,
\begin{equation} \label{eq:hyprN2}
\lim_{N \to \pinf} \frac{r_N}{N} = c > 0.
\end{equation}
Beware that, as we mentioned, this is a different rescaling from \eqref{eq:hyprN1} for the first model. We shall prove the following, where $\tc := c / \s$ and we recall that $\mu$ is a measure on $(0,\pinf)$ with density $1/(2 x^{3/2})$.

\begin{thm} \label{th:conv2}
Under the assumption \eqref{eq:hyprN2}, the sequence $(\PN_N(r_N))$ converges in distribution in $\cM^+$ to a random measure $\cP$, characterized by a cumulant $\k(f)$, unique solution to
\begin{equation} \label{eq:cumu2}
\exp - \k(f) = \l \int_0^{\pinf} \left ( 1 - e^{-f(x)} \right ) \mu(\ddx) + \frac1c \E \left ( 1 - e^{-f(P) - \k(f) C} \right )
\end{equation}
for every $f \in C_K$, where $(P,C)$ is a couple of random variables with Laplace transform
\begin{equation} \label{eq:lapltransPC}
\E(\exp - \a P - \b C) = \left ( \frac{\sqrt{2 \a} \tc}{\sinh \sqrt{2 \a} \tc} \right )^2 \frac{1}{\b c + \sqrt{2 \a} \tc \coth \sqrt{2 \a} \tc}.
\end{equation}
\end{thm}

We shall follow the same route as for Theorem \ref{th:conv1}: write an equation for the cumulant of $\PN(r_N)$, study its convergence, and prove that it has a unique solution. Obviously, we also need to prove the tightness, but it is obtained as for Lemma \ref{lem:tightness}.

\subsection{Interpretation} \label{sec:interpretation}

As for the first model, let us try to give a reasonable interpretation of the limiting measure $\cP$. First, let us study $(P,C)$. To this end, notice that we may rewrite
\[
\E(\exp - \a P - \b C) = \int_0^{\pinf} \frac{1}{c} e^{-x/c} e^{- \b x} \left ( \frac{\sqrt{2 \a} \tc}{\sinh \sqrt{2 \a} \tc} \right )^2  \exp - x \frac{1}{c} \left ( \frac{\sqrt{2 \a} \tc}{\tanh \sqrt{2 \a} \tc} - 1 \right ) \dx.
\]
This makes clear the following facts.
\begin{itemize}
	\item $C$ has an exponential $\cE(1/c)$ law (surprisingly, this quantity does not depend on the second moment of $\rho$).
	\item Conditionally on $C$, $P$ has the law of the sum of three independent variables: 
		\begin{itemize}
		\item two whose law has Laplace transform
		\[
		\frac{\sqrt{2 \a} \tc}{\sinh \sqrt{2 \a} \tc};
		\]
		\item a third one whose conditional Laplace Transform is
		\[
		\exp - C \frac{1}{c} \left ( \frac{\sqrt{2 \a} \tc}{\tanh \sqrt{2 \a} \tc} - 1 \right ).
		\]
		\end{itemize}
\end{itemize}

Note, and the reason for its appearance will be clear in the proofs, that $\sqrt{2 \a} \tc / \sinh \sqrt{2 \a} \tc$ is the Laplace transform of the law of the hitting time of $\tc$ by a Bessel 3 process. This distribution has a complicated density given by Theta functions, and we shall thus not dwell on this matter.

Let us then construct a random measure $\eta$ as follows. We first build a tree $T$ with fertilities, which is nothing else than a Galton-Watson tree with types in $\R^+$. Each individual of fertility $\f$ has a number of descendants distributed as a Poisson distribution with parameter $\f$. Each of these children chooses independently a fertility which is exponential $\cE(1/c)$. Conditionally on its fertility $\f$, we attach to each individual a population $P(\f)$ distributed as a variable with Laplace transform
\[
\E(\exp - \a P(\f)) = \left ( \frac{\sqrt{2 \a} \tc}{\sinh \sqrt{2 \a} \tc} \right )^2  \exp - \f \frac{1}{c} \left ( \frac{\sqrt{2 \a} \tc}{\tanh \sqrt{2 \a} \tc} - 1 \right ).
\]
We start from a (virtual) individual of fertility $1$ and population, say, 1. We have then constructed a tree $T$, to each vertex $v$ of which a fertility $\f_v$ and a population $P_v$ is attached. One readily checks that the reproduction law is critical and not $\dl_1$, so this tree is almost surely finite.

Then, consider $(\eta_u(\f), \f \geq 0)_{u \in \cU}$ a family of independent variables, such that $\eta_u(\f)$ is a Poisson random measure with intensity $\f \l \mu$. Define
\[
\eta = \sum_{u \in T} (\eta_u(\f_u) + \dl_{P_u}) - \dl_1.
\]
Subtracting $\dl_1$ just means that we do not take into account the population of the virtual initial individual. The reader will check, as for the first model, that the cumulant of $\eta$ solves the same equation \eqref{eq:cumu2} as the cumulant of $\cP$, so that these two measures have the same law.

\subsection{Equation of the cumulant}

Let us first notice that, as for the first model, the cumulant $\k_N(f)$ of $\PN(r_N)$, given by
\[
\k_N(f) = - \ln \E \left ( \exp - \la \PN_N(r_N) , f \ra \right ) = - N \ln \E \left ( \exp - \la \PN(r_N) , f \ra \right ),
\]
for $f \in C_K$, also solves Equation \eqref{eq:cumuN}. The proof is identical.

\begin{lemma} \label{lem:eqcumu2}
The cumulant $\k_N(f)$ solves the following equation
\begin{equation}\label{eq:cumuN2}
\exp - \k_N(f) = \E \left ( \exp - \left ( f \left ( \frac{P_N}{N^2} \right ) + \frac{C_N}{N} \k_N(f) \right ) \right )^N
\end{equation}
for every $f \in C_K$.
\end{lemma}

We now assume that $\k_N(f) \to \k(f)$. Once again, the RHS of \eqref{eq:cumuN2} has the same limit, if any, as
\begin{align*}
N \E \left ( 1 - \exp - \left ( f \left ( \frac{P_N}{N^2} \right ) \right. \right. & \left. \left. + \frac{C_N}{N} \k_N(f) \right ) \right ) \\
& = N \E \left ( \left ( 1 - \exp - f \left ( \frac{P_N}{N^2} \right ) \right ) \unn{C_N = 0} \right ) \\
& \quad + N \P(C_N > 0) \E \left ( \left. 1 - \exp - \left ( f \left ( \frac{P_N}{N^2} \right ) + \frac{C_N}{N} \k_N(f) \right ) \right \vert C_N > 0 \right ).
\end{align*}
The first part is easy to deal with. Note indeed that by the Cauchy-Schwarz inequality,
\[
N \E \left ( \left ( 1 - \exp - f \left ( \frac{P_N}{N^2} \right ) \right ) \unn{C_N = 0} \right )^2 \leq N \P(C_N = 0) \E \left ( \left ( 1 - \exp - f \left ( \frac{P_N}{N^2} \right ) \right )^2 \right ).
\]
By \eqref{eq:probacol} and \eqref{eq:hyprN2}, the first term in the RHS is bounded. By Lemma \ref{lem:popcolRW2}, $P_N \leq \vsinf$, so $P_N / N^2 \to 0$ a.s. and the second term thus tends to 0 by dominated convergence. Hence
\[
N \E \left ( \left ( 1 - \exp - f \left ( \frac{P_N}{N^2} \right ) \right )\unn{C_N = 0} \right ) \sim N \E \left ( 1 - \exp - f \left ( \frac{P_N}{N^2} \right ) \right ) \to \int_0^{\pinf} (1-e^{-f(x)}) \mu(\ddx)
\]
as for Lemma \ref{lem:convmeas}. On the second hand, note that by \eqref{eq:probacol} and \eqref{eq:hyprN2}
\begin{align*}
& N \P(C_N > 0) \E \left ( \left. 1 - \exp - \left ( f \left ( \frac{P_N}{N^2} \right ) + \frac{C_N}{N} \k_N(f) \right ) \right \vert C_N > 0 \right ) \\
& \sim \frac1c \E \left ( \left. 1 - \exp - \left ( f \left ( \frac{P_N}{N^2} \right ) - \frac{C_N}{N} \k_N(f) \right )\right \vert C_N > 0 \right ).
\end{align*}
We are thus led to compute the weak limit of the couple of random variables $(P_N/N^2,C_N/N)$ knowing that $C_N > 0$. This is more involved than for our first model, and this will be the goal of the next section.

\subsection{Population and number of colonies of a fertile island}

We shall now give a formula for the population and number of colonies of a fertile island, by proving the following result.

\begin{prop} \label{prop:convPNCN}
Conditionally on $\{ C_N > 0 \}$, the variable $(P_N/N^2,C_N/N)$ converges in law to a variable $(P,C)$ whose Laplace transform is given by \eqref{eq:lapltransPC}.
\end{prop}

\begin{remark}
In the proof, we shall actually assume that $r_N = \lfloor c N \rfloor$. The results hold without this restriction, but writing the proofs would then require some more pages of technical convolutions on which we do not wish to dwell. The essential ingredient to get rid of this restriction is to sandwich our processes, depending on $r_N/N$, between two processes depending on $c-\eps$ and $c+\eps$, where $\eps > 0$ can be chosen arbitrarily small, and to see that these two processes actually converge to the process depending on $c$ when $\eps \to 0$. Depending on the cases, this is due to the absolute continuity of the laws considered or the continuity of the local time in the space variable.
\end{remark}

The proof of this result will be twofold. First, $(P_N/N^2,C_N/N)$ can be reformulated, thanks to Lemma \ref{lem:popcolRW2}, in terms of functionals of a critical random walk with a second moment. At the limit, this quantity can thus be written as a functional of the Brownian motion. More precisely, we wish to condition on $C_N > 0$, which corresponds to conditioning this random walk to hit $[r_N, \pinf)$ before $-1$. The limit will thus actually be in terms of a functional of a Brownian excursion, conditioned on hitting $(c,\pinf)$. The second part of the proof is to compute the Laplace transform of this functional, which can be done through Williams' decomposition theorem of the excursion, along with a Ray-Knight theorem and some results of Pitman and Yor \cite{PitmanYor} concerning the Laplace transforms of functionals of Bessel bridges.

\subsubsection{Convergence to the excursion measure}

Let us first write $P_N$ and $C_N$ in terms of a functional of the random walk $(S_n)$, and to make this precise, we shall introduce some more notation. Let $\cE$ be the space of excursions, that is of nonnegative c\`adl\`ag functions on $\R^+$ such that, for every $\bfe \in \cE$,
\[
\zeta(\bfe) := \sup \{t \geq 0, \, \bfe_t > 0 \} \in [0,\pinf).
\]
This space is endowed with the distance
\[
\dl(\bfe,\bfe') = \sup_{t \geq 0} |\bfe_t - \bfe'_t| + |\zeta(\bfe) - \zeta(\bfe')|,
\]
which makes it a Polish space, see \cite{LeGallExcursion}. Let $\D^0$ be the space of c\`adl\`ag functions on $\R_+$ vanishing at $0$. For $d > 0$ and $f \in \D^0$, recall that $\tau_d(f)$ is the hitting time of $(d,\pinf)$, and let
\[
\tau_{d,-}(f) = \inf \{ t \leq \tau_d(f), \, f(t) = 0 \; \& \; \forall s \in [t,\tau_d) \, f(s) \geq 0 \}
\]
and finally
\[
\tau_{d,+}(f) = \inf \{ t \geq \tau_d(f), f(t) < 0 \}.
\]
These quantities are depicted in Figure \ref{fig:exc_full}.

\begin{figure}[htb]
\centering
\includegraphics[width= \columnwidth]{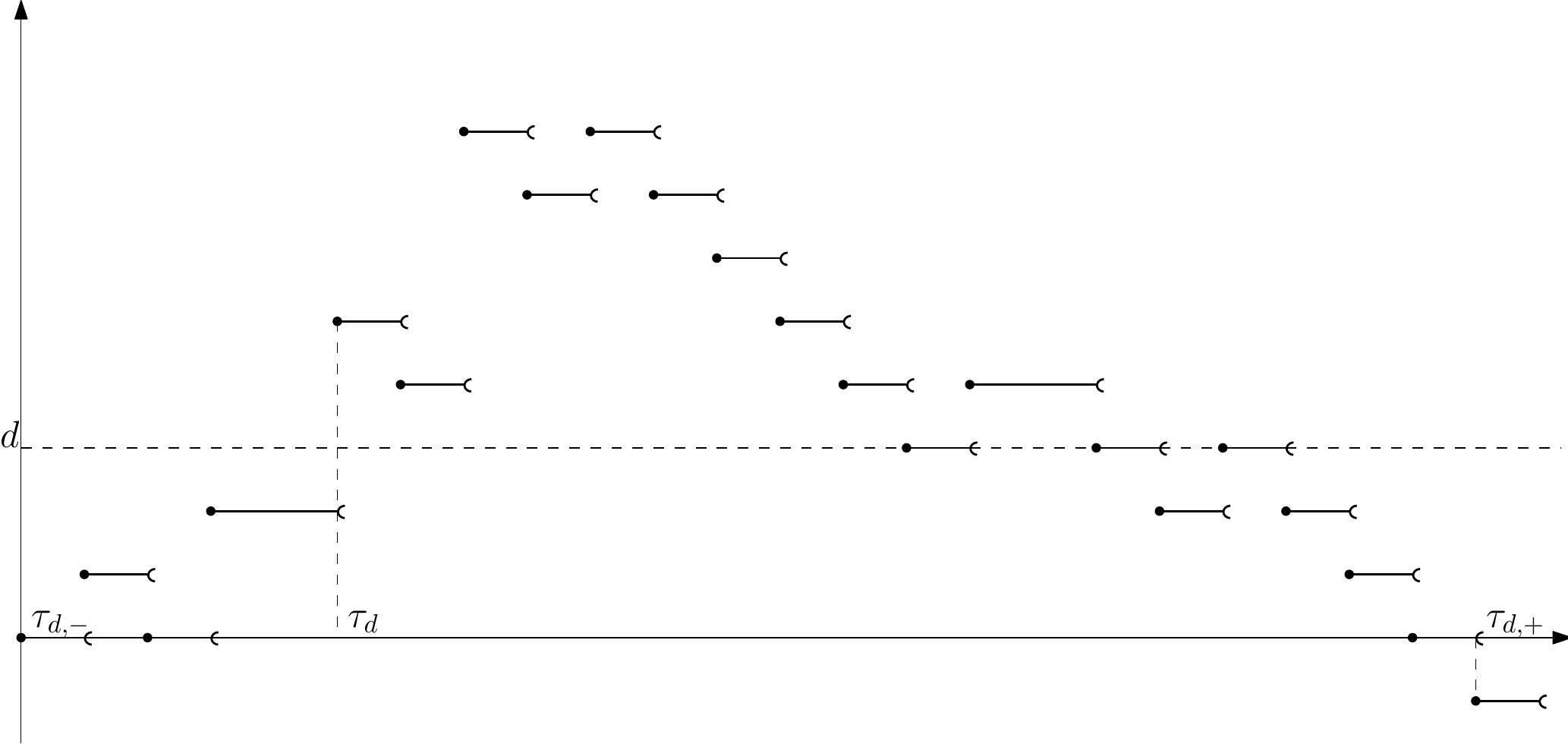}
\caption{The quantities $\tau_{d,-}(f)$, $\tau_d(f)$ and $\tau_{d,+}(f)$ shown on the exploration process of Figure \ref{fig:cont_expl_proc}.}
\label{fig:exc_full}
\end{figure}

We let $\D^0_d$ be the subset of $\D^0$ such that $\tau_d$ and $\tau_{d,+}$ are finite. For $(X,Y) \in \D^0_d \times \D^0$, we define $e_d$ the function which extracts from $X$ the first excursion which goes above level $d$, and shifts $Y$ accordingly, in formulas:
\[
e_d(X,Y) = \left ( X_{(\tau_{d,-}(X)+t)\wedge \tau_{d,+}(X)} \vee 0,Y_{(\tau_{d,-}(X)+t)\wedge \tau_{d,+}(X)} - Y_{\tau_{d,+}(X)} \right )_{t \geq 0}.
\]
Clearly, $e_d(X,Y) \in \cE \times \D^0$. Let us finally define, for $(X,Y) \in \cE \times \D^0$, 
\[
\phi_d(X,Y) = \left ( \int_0^{\zeta(X)} \unn{X_s < d} \ds, Y_{\zeta(X)} \right ) \in \R_+ \times \R,
\]
which computes the total time spent by $X$ below level $d$, as well as the value of $Y$ at the final point of $X$. This is probably a good time to state the important feature of these mappings. We let here $\W$ be the law of the Brownian motion (which is in particular a law on $\D^0_d$).

\begin{lemma} \label{lem:contmap}
\begin{itemize}
\item For $\W$-a.e. $X$ and every $Y \in \D^0$, the mapping $e_d : \D^0_d \times \D^0 \to \cE \times \D_0$ is continuous at $(X,Y)$.
\item For any $X \in \cE$ and continuous $Y \in \D^0$, the function $\phi_d$ is continuous at $(X,Y)$.
\end{itemize}
\end{lemma}

\begin{proof}
The first part of the statement just stems from the fact that a.s., when hitting a value, the Brownian motion oscillates around it. Hence, the times $\tau_{d,-}$, $\tau_d$ and $\tau_{d,+}$ are continuous at $\W$-almost every $X$, and the continuity of $e_d$ then readily follows from the continuity of the Brownian motion. The second part is trivial by definition of the distance $\dl$ and dominated convergence.
\end{proof}

From Lemma \ref{lem:popcolRW2}, it is now natural to define
\[
L_t = \sum_{i, \, \vs_i \leq t} \cO^N_i(r_N)
\]
be the sum of the overshoots above level $r_N$ of $(S_n)$ up to time $t \geq 0$. Now, let us define our rescaled random walk
\[
\SN_t = \frac{1}{\s N} S_{N^2t}, \quad t \geq 0,
\]
and the rescaled overshoot process
\[
\LN_t = \frac{1}{\s N} L_{N^2t}, \quad t \geq 0.
\]
Recall that $\tc = c / \s$. We may then reformulate Lemma \ref{lem:popcolRW2}, by stating that, whenever $C_N > 0$,
\begin{equation} \label{eq:PNCN}
\left. \left ( \frac{1}{N^2} P_N, \frac1N C_N \right ) \right | \{ C_N > 0 \} \eqlaw \phi_{\tc}(e_{\tc}(\SN,\s \LN)).
\end{equation}
We are now in a good position to state the main result of this section, which is quite similar to the main result of \cite{LeGallExcursion}. To this end, recall from Williams' description of It\^o's measure (see e.g. Th. 4.5 p. 499 in \cite{RY}) that the It\^o measure of the set of excursions with a maximum greater than $\tc$ is finite. Hence, we may define $\n^{> \tc}$ a probability law on $\cE$, which is the law of an excursion conditioned on having a maximum greater than $\tc$. This is in particular a semi-martingale, and we can thus define its local time. More generally, for a semi-martingale $X$, we let $\ell^a_t(X)$ is its local time at level $a$ up to time $t$, and we obviously consider a modification of $(\ell^a_t(X), a \in \R, t \geq 0)$ which is a.s. continuous in $t$ and c\`adl\`ag in $a$, see \cite{RY}.

\begin{lemma} \label{lem:limPNCN}
The convergence in distribution
\begin{equation} \label{eq:limPNCN}
\left. \left ( \frac{1}{N^2} P_N, \frac1N C_N \right ) \right | \left \{ C_N > 0 \right \} \to \left ( \int_0^{\zeta(\bfe)} \unn{\bfe_s < \tc} \ds, \frac{\s}{2} \ell^{\tc}_{\infty}(\bfe) \right )
\end{equation}
holds, where $\bfe$ has law $\n^{> \tc}$.
\end{lemma}

\begin{proof}
This is obviously obtained by passing to the limit in \eqref{eq:PNCN}, though this requires some care. It should be intuitively clear that $(\SN,\LN)$ converges to a Brownian motion and its local time at $\tc$, which is precisely the content of Theorem 1.3 of \cite{PerkinsLT}, which we may reformulate, if we are careful of the different normalization from \cite{RY}, as
\[
(\SN,\LN) \to \left ( B,\frac12 \ell^{\tc}(B) \right )
\]
weakly in $D([0,\pinf),\R^2)$, where $B$ is a standard Brownian motion. One should just take note of the two following facts:
\begin{itemize}
\item our walk is left-continuous, so the definition of $\LN$ we give is precisely Formula (1.2) in \cite{PerkinsLT}, with $x = \tc$;
\item the only slight difference is that Theorem 1.3 of \cite{PerkinsLT} deals with versions of the random walk and the local time which are linearly interpolated, unlike ours, but since the limit is continuous, and hence the limit holds for the topology of uniform convergence on the compacts, it clearly does not make any difference.
\end{itemize}

Now, the latter, the continuity of $e_{\tc}$ and the continuous mapping theorem ensure that
\[
e_{\tc}(\SN,\s \LN) \to e_{\tc} \left ( B,\frac{\s}{2} \ell^{\tc}(B) \right )
\]
weakly in $\cE \times \D_0$. But the definition of $e_{\tc}$ ensures that the first coordinate of $e_{\tc}(B,\s \ell^{\tc}(B)/2)$ is the first excursion of a Brownian motion which goes above level $\tc$, and has thus law $\n^{> \tc}$. The second coordinate is $\s/2$ times its total local time at $\tc$. The result then follows from the continuity of $\phi_{\tc}$.
\end{proof}

\subsubsection{Time spent under a level and local time of an excursion}

The last part of the proof of Proposition \ref{prop:convPNCN} is then to compute the Laplace transform of the RHS of \eqref{eq:limPNCN}. Hence, from now on, $\bfe$ is an excursion conditioned on having a maximum greater than $\tc$, that is a process with law $\n^{> \tc}$. We let
\[
(P,C) = \left ( \int_0^{\zeta(\bfe)} \unn{\bfe_s < \tc} \ds, \frac{\s}{2} \ell_{\infty}^{\tc}(\bfe) \right ).
\]

\begin{lemma}
The variable $(P,C)$ has Laplace transform given by \eqref{eq:lapltransPC}.
\end{lemma}

\begin{proof}
Let us explain how to construct such an excursion $\bfe$. From Williams' decomposition (see \cite[Th 4.5 p. 499]{RY}), conditionally on its maximum $M$, which has a ``law'' with density $1/(2x^2) \unn{x > 0}$, the Brownian excursion has the law of two independent Bessel 3 processes, until they hit $M$, put back-to-back. Hence, if $R$ is a Bessel 3 process, then
\begin{equation} \label{eq:williams}
\E(\exp - \a P - \b C) = 2 \tc \int_{\tc}^{\pinf} \frac{1}{2x^2} \E \left ( \exp - \a \int_0^{\tau_x(R)} \unn{R_s < \tc} \ds - \frac{\b \s}{2} \ell_{\tau_x(R)}^{\tc}(R) \right )^2 \dx
\end{equation}
and we are thus led to compute, for a fixed $x \geq \tc$,
\[
g(\a,\b) := \E \left ( \exp - \a \int_0^{\tau_x(R)} \unn{R_s < \tc} \ds - \frac{\b \s}{2} \ell_{\tau_x(R)}^{\tc} \right ).
\]
Now, recall that the Bessel processes have the Brownian scaling property, i.e. $R$ has the same law as $Q=(x R_{t/x^2})_{t \geq 0}$, so the occupation time formula provides
\[
\ell^a_{\tau_x(Q)}(Q) = x \ell^{a/x}_{\tau_1(R)}(R)
\]
and thus, still by Brownian scaling,
\begin{align*}
g(\a,\b) & = \E \left ( \exp - \a \int_0^{\tau_x(Q)} \unn{Q < \tc} \ds - \frac{\b \s}{2} \ell_{\tau_x(Q)}^{\tc}(Q) \right ) \\
& = \E \left ( \exp - \a \int_0^{x^2 \tau_1(R)} \unn{R_{s/x^2} < \tc/x} \ds - \frac{\b \s}{2} x \ell^{\tc/x}_{\tau_1(R)}(R) \right ) \\
& = \E \left ( \exp - \a x^2 \int_0^{\tau_1(R)} \unn{R_s < \tc/x} \ds - \frac{\b \s}{2} x \ell^{\tc/x}_{\tau_1(R)}(R) \right ).
\end{align*}
Hence, we are now led to compute, for $u,v \geq 0$ and $0 \leq r \leq 1$,
\[
h(u,v) := \E \left ( \exp - u \int_0^{\tau_1(R)} \unn{R_u < r} \du - v \ell_{\tau_1(R)}^r(R) \right ).
\]
By the occupation time formula (note that the quadratic variation of a Bessel process is $t$),
\[
h(u,v) = \E \left ( \exp - u \int_0^r \ell_{\tau_1(R)}^u(R) \du - v \ell_{\tau_1(R)}^r(R) \right ),
\]
which we may rewrite as
\[
\E \left ( \exp - \int_0^1 \ell_{\tau_1(R)}^x(R) \nu(\ddx) \right )
\]
where $\nu(\ddx) = u \un_{[0,r]}(\ddx) + v \dl_r$. From a Ray-Knight Theorem\footnote{Surprisingly hard to find in the literature: it can be seen as a consequence of Theorem 4 in \cite{WilliamsDecomp} along with the Williams' decomposition, Theorem 3.11 in \cite{RY}. This statement is also given in \cite{MansuyYor}, p.42, along with a direct proof.}, the process $(\ell_{\tau_1}^x(R), x \in [0,1])$ is a squared Bessel 2 bridge, and thus, the latter quantity is precisely computed in Prop. 5.10 in \cite{PitmanYor}. After easy but long computations, one obtains
\[
\E \left ( \exp - \int_0^1 \ell_{\tau_1(R)}^x(R) \nu(\ddx) \right ) = \frac{\sqrt{2 u}}{(1-r)(2 v \sinh \sqrt{2 u} r + \sqrt{2 u} \cosh \sqrt{2 u} r) + \sinh \sqrt{2 u} r}.
\]
Hence
\begin{align*}
g(\a,\b) & = \frac{x \sqrt{2 \a}}{(1-\tc/x)(\b \s x \sinh \sqrt{2 \a} \tc + x \sqrt{2 \a} \cosh \sqrt{2 \a} \tc) + \sinh \sqrt{2 \a} \tc} \\
& = \frac{\sqrt{2 \a}}{\sinh \sqrt{2 \a} \tc} \frac{x}{1 + (x-\tc)(\b \s + \sqrt{2 \a} \coth \sqrt{2 \a} \tc)}.
\end{align*}
The result then readily follows from plugging this formula in \eqref{eq:williams}.
\end{proof}

\begin{remark}
Recall the interpretation we gave of the variable $(P,C)$ in Section \ref{sec:interpretation}. This formulas can thus be interpreted by saying that the excursion can be split into three parts: a part between 0 and the hitting time of $\tc$, which has the law of a Bessel 3 process. Another part between the last hitting time of $\tc$ and 0, with the same law. And in between, there is an exponential ``quantity'' of excursions above and below $\tc$; each contributes independently to a microscopic random amount, given by the above Laplace transform, to the time spent below $\tc$. See also the next section for a discrete version of these heuristics.
\end{remark}

Proposition \ref{prop:convPNCN} then follows from the result just proven and Lemma \ref{lem:limPNCN}. To conclude the proof of Theorem \ref{th:conv2}, all we need to check now is that Equation \eqref{eq:cumu2} has a unique solution which is done as for Theorem \ref{th:conv1}.

\subsection{Another way to the result}

To conclude, let us present another way to compute the limit of $(P_N/N^2,C_N/N)$ knowing that $C_N > 0$. This method is more elementary but requires more steps, and might be seen as more natural -- at least to the author, who firstly used it to derive the result. Once again, the goal is to compute the time spent under $r_N$, and the sum of the overshoots, both up to the hitting time of $-1$, for the random walk $(S_n)$ conditioned on hitting $[r_N,\pinf)$ before $-1$.

To obtain such an excursion of the random walk, we wait until we see $(S_n)$ hit $[r_N,\pinf)$, and we consider the excursion from $0$ to $-1$ straddling this time. We can cut this excursion in three pieces:
\begin{itemize}
	\item a first piece, where the walk goes from 0 to $[r_N,\pinf)$;
	\item a third piece, where the walk goes from $r_N$ to $-1$;
	\item a second piece in between these two, where the walk goes from $r_N$ to $r_N$ without going back to 0, a certain amount of time.
\end{itemize}
By the same reasonings as above, it should be clear that the first piece, after rescaling, converges to the first piece of a Brownian motion going from 0 to $\tc$ while remaining positive. From Williams' decomposition, Theorem 3.11 in \cite{RY}, this has the law of a Bessel 3 process $R$, and thus the time it takes to hit $\tc$ has the law of $\tau_{\tc}(R)$, which has, as we mentioned, Laplace transform
\[
\frac{\sqrt{2 \a} \tc}{\sinh \sqrt{2 \a} \tc}.
\]
The third part obviously accounts for the same independent quantity.

Now, let us study what happens in the middle. We may forget\footnote{However, this relies on the walk having a second moment. Otherwise, it could have a macroscopic jump from below $r_N$ to above.} about the time for the random walk to go from $[r_N,\pinf)$ back to $r_N - 1$, and first consider the walk at $r_N - 1$. From Formulas (3) p. 187 and (b) p. 181 in \cite{Spitzer}, the size of an overshoot has finite mean $\s^2/2$, so \eqref{eq:hyprN2} and a gambler's ruin estimate (see Lemma 5.1.3 in \cite{LawlerLimic}) show that the walk goes to $-1$ before coming back at $r_N$ with probability $\s^2/(2cn)$. Hence, if we let $L_N$ to be the number of such excursions from $r_N$ to $r_N$ without going back to $-1$, we deduce therefrom that $L_N / N$ converges to a variable $L$ with an exponential law $\cE(\s^2/(2c))$.

Now, conditionally on $L$, i.e., loosely, $L_N \approx L N$, the number of colonies $C_N$ is the sum of $L_N$ independent overshoots, and thus, by the law of large numbers, $C_N/N$ converges to $L$ times the mean size of an overshoot, i.e. $C = \s^2 L /2$.

Finally, still conditionally on $L$, we want to know the time spent below $r_N$ by this piece of the walk. At the limit, after rescaling and by similar reasonings as in the above section, it can be seen as the time spent below 0 by a Brownian motion
\begin{itemize}
	\item up until it has accumulated a local time $\s L$ at 0,
	\item and conditioned on not hitting $- \tc$ before this time (what has positive probability).
\end{itemize}
We thus have to compute this quantity. But to construct a Brownian motion up to a local time of $\ell$, all we need is a Poisson point process $(e_s)_{s \in [0,\ell]}$ with intensity $\n$, the It\^o measure, and then glue these excursions together (see \cite{RY}). Hence, by thinning, to construct such a conditioned Brownian motion, we take a Poisson point process $(e_s)_{s \in [0,\ell]}$ with intensity $\n_{> - \tc}$, the restriction of $\n$ to the excursions with minimum greater than $- \tc$, and then glue these excursions together. Let $R(e)$ be the length of an excursion $e$. Then the time spent by this conditioned Brownian motion under $0$ is
\[
\sum_{0 \leq s \leq \ell} R(e_s) \unn{\inf e < 0}
\]
and, by symmetry and using the exponential formula for Poisson point processes (\cite{RY}, p. 476), for $\a \geq 0$,
\begin{align*}
\E \left ( \exp - \a \sum_{0 \leq s \leq \ell} R(e_s) \unn{\inf e < 0} \right ) & = \exp - \ell \int (1 - e^{- \a R(u)}) \unn{ \inf u < 0} \: \n_{ > - \tc}(\ddu) \\
& = \exp - \ell \int (1 - e^{- \a R(u)}) \unn{ 0 < \sup u < \tc} \: \n(\ddu) \\
& = \exp - \ell \int_{x = 0}^{\tc} \int_{y=0}^{\pinf} (1 - e^{- \a y}) \: \m(\ddx,\ddy)
\end{align*}
where $\m$ is the image of $\n$ (or $\n_+$, the It\^o measure of the unsigned excursion) by $e \mapsto (\sup e,R(e))$. Once again, Williams' decomposition of the Brownian excursion tells that its maximum $M$ has ``law'' $1/(2x^2) \unn{x > 0} \dx$, and that conditioned on this maximum, it has the law of two independent Bessel 3 processes $R$ and $R'$ put back to back. Therefore
\begin{align*}
\int_{x = 0}^{\tc} \int_{y=0}^{\pinf} \left ( 1 - e^{- \a y} \right ) \: \m(\ddx,\ddy) & = \int_{x = 0}^{\tc} \frac{1}{2 x^2} \E \left ( \left ( 1 - e^{- \a (\tau_x(R) + \tau_x(R'))} \right ) \right ) \dx \\
& = \int_{x = 0}^{\tc} \frac{1}{2 x^2} \left ( 1 - \E(e^{- \a \tau_x(R)})^2 \right ) \dx \\
& = \int_{x = 0}^{\tc} \frac{1}{2 x^2} \left ( 1 - \frac{2 \a x^2}{\sh^2 x \sqrt{2 \a}} \right ) \dx \\
& = \frac{1}{2 \tc} \left ( \frac{\sqrt{2 \a} \tc}{\tanh \sqrt{2 \a} \tc } - 1 \right ).
\end{align*}
Putting the pieces together, this is just saying that, conditionally on $L = 2 C / \s^2$, the variable $P$ has Laplace transform
\[
\left ( \frac{\sqrt{2 \a} \tc}{\sinh \sqrt{2 \a} \tc} \right )^2 \exp - \s L \frac{1}{2 \tc} \left ( \frac{\sqrt{2 \a} \tc}{\tanh \sqrt{2 \a} \tc } - 1 \right ) = \left ( \frac{\sqrt{2 \a} \tc}{\sinh \sqrt{2 \a} \tc} \right )^2 \exp - C \frac1c \left ( \frac{\sqrt{2 \a} \tc}{\tanh \sqrt{2 \a} \tc } - 1 \right )
\]
which is precisely what we remarked after Theorem \ref{th:conv2}. The main advantage of this method is thus probably that it provides directly the law of $P$ knowing $C$, which is not obvious when merely looking at Formula \eqref{eq:lapltransPC}.

\bibliographystyle{abbrv}
\bibliography{Bibli}

\end{document}